\documentclass[12pt]{amsart}
 \newtheorem{thm}{Theorem}[section]
 \newtheorem{cor}[thm]{Corollary}
 
 \newtheorem{prop}[thm]{Proposition}
 \theoremstyle{definition}
 
 \theoremstyle{remark}
 \newtheorem{rem}[thm]{Remark}
 \newtheorem{ex}[thm]{Example}
 \numberwithin{equation}{section}

  \usepackage{amsmath,amssymb}
  \usepackage{amsthm}
\usepackage{mathrsfs}
\usepackage{enumerate}
\usepackage{graphicx}
\usepackage{tikz}\parindent=0.cm
\usetikzlibrary{shapes.geometric}
\usetikzlibrary{positioning}
\usetikzlibrary{arrows}
\usetikzlibrary{shapes.multipart}
\usepackage{physics}
\usepackage[notcite, final, notref]{showkeys}
\usepackage{dsfont}

\usepackage{shadow}
\usepackage{amscd}
\newcommand{\e}{\varepsilon}

\usepackage{tikz}
  \usepackage{tikz-cd} 
\usetikzlibrary{arrows,shapes,positioning,shadows,trees,matrix}
\usepackage{xcolor}
\usepackage{pst-node}

\begin{document}

\title[Trace formula]{Operator model and a trace formula for pairs of unitary operators}

\author[D. Alpay]{Daniel Alpay}
\address{(DA) Schmid College of Science and Technology \\
Chapman University\\
One University Drive
Orange, California 92866\\
USA}
\email{alpay@chapman.edu}

\author[F. Colombo]{Fabrizio Colombo}
\address{(FC) Politecnico di
  Milano\\Dipartimento di Matematica\\Via E. Bonardi, 9\\20133 Milano\\Italy}
  \email{fabrizio.colombo@polimi.it}

  \author[I. Lewkowicz]{Izchak Lewkowicz}
  \address{(IL) School of Electrical and Computer Engineering\\
Ben-Gurion University of the Negev\\ P.O.B. 653\\ Beer-Sheva, 84105\\
Israel}
\email{izchak@bgu.ac.il}

\author[I. Sabadini]{Irene Sabadini}
\address{(IS) Politecnico di
  Milano\\Dipartimento di Matematica\\Via E. Bonardi, 9\\20133 Milano\\Italy}
\email{irene.sabadini@polimi.it}

\thanks{
           Daniel Alpay thanks the Foster G. and Mary McGaw Professorship in Mathematical Sciences, which supported this research}

         \begin{abstract}
           Using the theory of reproducing kernel Hilbert spaces introduced by L. de Branges and J. Rovnyak we prove a trace formula for pairs of operators in Hilbert space
           in terms of a Carath\'eodory function. We consider the special case, when the latter is rational. An
           application to the theory of first order discrete systems is given.
\end{abstract}
\maketitle
\tableofcontents



\section{Introduction}
\setcounter{equation}{0}
In \cite{dbr1} L. de Branges and J. Rovnyak give a model for pairs of self-adjoint operators in terms of reproducing kernel Hilbert spaces associated to a function analytic and
with a positive real part in the open upper half plane  (i.e. a Nevanlinna function). This construction was used in \cite{MR1960423,MR2069002} to prove a related trace formula.
In the case of pairs of unitary operators a model was given in \cite{dbs} by L. de Branges and L.A. Shulman in terms of a function analytic and with a positive real part in the open
unit disk $\mathbb D$. Such functions are called  Carath\'eodory functions in the mathematical literature and discrete-time-positive-real functions
in the engineering literature.\smallskip

In the present work we prove a trace formula for a pair of unitary operators. We use a model for pairs of unitary operators different from the one appearing in \cite{dbs},
still in terms of a Carath\'eodory function, and developed by one of the authors with V. Bolotnikov in \cite[\S 5.3 p. 104]{MR1638044}
(see \cite[\S 5.3 p. 87]{MR2002b:47144} for the English translation). We give an application to the theory of first order systems, appearing in
particular in the theory of orthogonal polynomials on the unit circle; see for instance \cite{DI1, szego} and the references therein.\smallskip

The paper consists of seven sections besides the introduction. In Section \ref{sec-2} we review the theory of reproducing kernel Hilbert spaces
associated to Carath\'eodory functions. The differences between the present approach and the one in \cite{dbs} is considered in Section \ref{2.1}.
The model for pairs of unitary operators is presented in Section \ref{sec-3}. Sections \ref{Tr1} and \ref{tr2} are devoted to two proofs
of the trace formula. The rational case is considered in Section \ref{ra-case}. In Section \ref{seccano} we illustrate the trace formula in the setting of canonical  discrete
systems.

\section{$\mathcal L(\Phi)$ spaces}
\setcounter{equation}{0}
\label{sec-2}
Let $\Phi$ be a $\mathbb C^{n\times n}$-valued function analytic in the open unit disk $\mathbb D$ and having  a positive real part there. Herglotz'
representation theorem (see \cite{hspnw,herglotz} for the original paper and \cite{MR48:904} for a proof in the operator-valued setting)  asserts that
\begin{equation}
\label{phi-herg}
\Phi({\lambda})=iA+\int_{0}^{2\pi}\frac{e^{it}+{\lambda}}{e^{it}-{\lambda}}d\mu(e^{it}),
\end{equation}
where $\mu$ is a $\mathbb C^{n\times n}$-valued positive measure on the unit circle $\mathbb T$ and $A\in\mathbb C^{n\times n}$ satisfies $A+A^*=0$.
Formula \eqref{phi-herg} allows to extend $\Phi$ to
\[
\mathbb E\stackrel{\rm def.}{=}\left\{{\lambda}\in\mathbb C\,\,;\,\, |{\lambda}|>1\right\}
\]
via the formula
\begin{equation}
  \label{2-2}
\Phi({\lambda})=-\left(\Phi\left(\frac{1}{\overline{{\lambda}}}\right)\right)^*,\quad {\lambda}\in\mathbb E.
  \end{equation}
For this extension one has the formula
\begin{equation}
  \label{ext-2-2}
  L_\Phi({\lambda},w)\stackrel{\rm def.}{=}    \frac{\Phi({\lambda})+(\Phi(w))^*}{2(1-{\lambda}\overline{w})}=\int_0^{2\pi}\frac{d\mu(e^{it})}{(e^{it}-{\lambda})(e^{-it}-\overline{w})},\quad {\lambda},w\in
  \mathbb C\setminus      \mathbb T.
    \end{equation}
    Thus the kernel $L_\Phi({\lambda},w)$ is positive definite in  $\mathbb C\setminus      \mathbb T$. We will denote by $\mathcal L(\Phi)$
    the associated reproducing kernel Hilbert space.\smallskip

    \begin{ex}
      \label{exexex}
    For instance, for $\Phi({\lambda})=1$ in $\mathbb D$ we get $\Phi(\lambda)=-1$ for ${\lambda}\in\mathbb E$ and the extended kernel $L_\Phi$  is
    \begin{equation}
      \label{1/2}
      L_\Phi(\lambda,w)=    \begin{cases}\,\, \dfrac{1}{1-\lambda\overline{w}},\quad \lambda,w\in\mathbb D,\\
        \\
        \,\, \dfrac{-1}{1-\lambda{\overline{w}}},\quad \lambda,w\in\mathbb E,\\  \\
        \,\, 0,\,\,\hspace{12.3mm}{\rm otherwise.}\end{cases}
      \end{equation}
This example serves also to show that in general there is no analytic continuation of the kernel across $\mathbb T$.
\end{ex}

\begin{thm}
A non-constant $\mathbb C^{n\times n}$-valued Carath\'eodory function $\Phi$ extended to $\mathbb E$ is meromorphic in the complex plane if and only if $\mu$ is a jump measure with a finite number of jumps, i.e. if and only if $\Phi$ is of the form
\begin{equation}
  \label{link}
\Phi(\lambda)=iE+\sum_{\ell=1}^mM_\ell \frac{e^{it_\ell}+\lambda}{e^{it_\ell}-\lambda},
\end{equation}
where $E\in\mathbb C^{n\times n}$ is Hermitian, $m\in\mathbb N$, $t_1,\ldots,t_m\in[0,2\pi)$ and $M_1,M_2,\ldots, M_m$ positive semi-definite $\mathbb C^{n\times n}$ matrices.
  \end{thm}

  \begin{proof}
It comes from the inversion formula which gives $\mu$ in terms of $\Phi$.
    \end{proof}

For the next theorem see for instance \cite[Lemma 6.5 p. 629]{ad1}. The fact that \eqref{ad1-1} is well defined follows from the fact that the closed linear span of the functions
of the form $c/(e^{it}-\lambda)$ where $\lambda\in\mathbb C\setminus\mathbb T$ and $c\in\mathbb C^n$ is dense in $\mathbf L_2^n(d\mu)$. This follows from the density in the
uniform norm of the rational functions in the set of continuous functions on $\mathbb T$. See \cite[Theorem 1 p. 192 and Theorem 2 p. 193]{zbMATH03263403} for the latter.

\begin{thm}
  Let $\Phi$ be a $\mathbb C^{n\times n}$-valued Carath\'eodory function, extended to $\mathbb E$ via \eqref{2-2}. The associated reproducing kernel Hilbert space
  $\mathcal L(\Phi)$ of functions analytic in $\mathbb C\setminus\mathbb T$ can be described as the set of functions of the form
      \begin{equation}
\label{rep-F-phi}
        \mathcal L(\Phi)=\left\{F({\lambda})=\int_0^{2\pi}\frac{d\mu(e^{it})f(e^{it})}{e^{it}-{\lambda}},\quad f\in \mathbf L_2^n(d\mu)\right\},
      \end{equation}
      with norm
      \begin{equation}
\label{ad1-1}
        \|F\|_\Phi=\|f\|_\mu.
        \end{equation}
      \end{thm}

      We denote by $W$ the map
      \begin{equation}
        Wf({\lambda})=F({\lambda})=\int_0^{2\pi}\frac{d\mu(e^{it})f(e^{it})}{e^{it}-{\lambda}},\quad f\in \mathbf L_2^n(d\mu).
    \end{equation}
    Furthermore, we denote by
      $R_a$ the backward-shift operator
    \begin{equation}
      (R_aF)({\lambda})=\begin{cases}\, \dfrac{F({\lambda})-F(a)}{{\lambda}-a},\quad {\lambda}\in V\setminus\left\{a\right\},\\
        \,\,\, F^\prime(a),\,\,\,\,\, \quad \hspace{0.86cm}{\lambda}=a,\end{cases}
      \end{equation}
      defined for $F$ analytic in a neighborhood $V$ of the point $a\in\mathbb C$.

      \begin{thm}
        The operator $S$ of multiplication by $e^{it}$ is unitary from $\mathbf L_2^n(d\mu)$ onto itself and
        \begin{equation}
          (          W(S-aI_{\mathbf L_2^n(d\mu)})^{-1}W^{-1}F)({\lambda})=(R_aF)({\lambda}),\quad a\in\mathbb C\setminus \mathbb T.
          \end{equation}
        \end{thm}

        \begin{proof}
          This follows from
          \[
            ( R_aF)({\lambda})=\int_0^{2\pi}\frac{d\mu(e^{it})f(e^{it})}{(e^{it}-{\lambda})(e^{it}-a)}.
          \]

        \end{proof}

        We note that the above result still holds for $a\in\mathbb T$ such that
\[
        {\rm Tr}\,\left(\int_0^{2\pi}\frac{d\mu(e^{it})}{|e^{it}-a|^2}\right)<\infty.
\]
We also note that the limit
\begin{equation}
  \lim_{\lambda\rightarrow\infty}\lambda F(\lambda)=-\int_0^{2\pi}d\mu(e^{it})f(e^{it})
\label{lim+}
\end{equation}
exists for all $F\in\mathcal L(\Phi)$.

\begin{thm}
  \label{circle}
  A reproducing kernel Hilbert space $(\mathcal H,[\cdot,\cdot])$ of $\mathbb C^n$-valued functions analytic in $\mathbb C\setminus\mathbb T$ has
  a reproducing kernel of the form $L_\Phi$ if and only if it is $R_a$-invariant for $a\in\mathbb C\setminus\mathbb T$ and if moreover, for every
  $a,b\in\mathbb C\setminus\mathbb T$ and $F,G\in\mathcal H$, it holds that
        \begin{equation}
          \label{structure-ide}
[F,G]+a[R_aF,G]+\overline{b}[F,R_bG]-(1-a\overline{b})[R_aF,R_bG]=0.
          \end{equation}
          In particular $\ker R_a=\left\{0\right\}$ for $a\in\mathbb C\setminus\mathbb T$.
                  \end{thm}

        \begin{proof}
The line version of the theorem appears in \cite[p. 13]{dbbook}, and the proof we present is adapted from that source. Let $K({\lambda},w)$ be the reproducing kernel of $\mathcal H$.
 We take $a\not=0$ and $\alpha$ two different points in $\mathbb C\setminus\mathbb T$ and
set in  \eqref{structure-ide}    $b=\frac{1}{\overline{a}}$ and     $G=K_\alpha c$ with $c\in\mathbb C^n$ (i.e. $G({\lambda})=K({\lambda},\alpha)c$). With $F\in\mathcal H$  we first note that
\[
\begin{split}
  [F,G]+a[R_aF,G]&=c^*F(\alpha)+c^*a\frac{F(\alpha)-F(a)}{\alpha-a}\\
  &=c^*\left(\frac{\alpha F(\alpha)-aF(a)}{\alpha-a}\right)\\
  &=[F,\frac{\overline{\alpha}}{\overline{\alpha}-\overline{a}}K_\alpha c-\frac{\overline{a}}{\overline{\alpha}-\overline{a}}K_a c].
  \end{split}
  \]
  Thus, \eqref{structure-ide} can be rewritten as
  \[
    [F, \frac{\overline{\alpha}}{\overline{\alpha}-\overline{a}}K_\alpha c-\frac{\overline{a}}{\overline{\alpha}-\overline{a}}K_a c+\frac{1}{\overline{a}}
    R_{\frac{1}{\overline{a}}}      K_\alpha c]=0,\quad\forall F\in\mathcal H,
    \]
    and hence we have
    \[
      \frac{\overline{\alpha}}{\overline{\alpha}-\overline{a}}      K({\lambda},\alpha) -\frac{\overline{a}}{\overline{\alpha}-\overline{a}}K({\lambda},a) +\frac{1}{\overline{a}}
\frac{K({\lambda},\alpha)-K(\frac{1}{\overline{a}},\alpha)}{{\lambda}-\frac{1}{\overline{a}}}=0.
  \]
  Thus
  \[
    \left(\frac{\overline{\alpha}}{\overline{\alpha}-\overline{a}}-\frac{1}{1-\overline{a}{\lambda}}\right)K({\lambda},\alpha)=
\frac{\overline{a}}{\overline{\alpha}-\overline{a}}K({\lambda},a) +\frac{1}{\overline{a}{\lambda}-1}K(\frac{1}{\overline{a}},\alpha).
    \]
    Equivalently
    \[
(1-{\lambda}\overline{\alpha})K({\lambda},\alpha)=\Phi_1({\lambda})+\Phi_2(\alpha)
      \]
      with
      \begin{equation}
          \Phi_1({\lambda})=(1-\overline{a}{\lambda})K({\lambda},a)\quad{\rm and}\quad
          \Phi_2(\alpha)=-\frac{  \overline{\alpha}-\overline{a}}{\overline{a}} K(\frac{1}{\overline{a}},\alpha).
        \end{equation}
        Since the kernel is Hermitian we have
        \[
\Phi_1({\lambda})+\Phi_2(\alpha)=(\Phi_1(\alpha))^*+(\Phi_2({\lambda}))^*,
\]
and it follows that $\Phi_2(\alpha)=(\Phi_1(\alpha))^*+M$, where $M$ is a constant Hermitian matrix. Hence, the kernel is of the required form with $\Phi({\lambda})=\Phi_1({\lambda})
+\frac{M}{2}$.\smallskip

          To prove the second claim let $F\in\ker R_a$ and let $F=G$ and $a=b$ in \eqref{structure-ide}.
          One gets $[F,F]=0$ and hence $F=0$.
          \end{proof}

          We illustrate this last claim on \eqref{1/2}. Let $f\in\ker R_0$, Then there exists constants $d$ and $e$ such that
            the function
            \[
              F({\lambda})=\begin{cases}\,\,d\quad {\lambda}\in\mathbb D,\\
                \,\,e\quad {\lambda}\in\mathbb E,\end{cases}
            \]
            belongs to $\mathcal L(\Phi)$. This is possible inside $\mathbb D$ since the kernel corresponds to the Hardy space of the open unit disk
            $\mathbf H_2(\mathbb D)$, but the restriction of
            the kernel to $\mathbb E$ corresponds to the space of functions of the form
\[
  F({\lambda})=\frac{1}{{\lambda}}h(1/{\lambda})
\]
with $h\in\mathbf H_2(\mathbb D)$ and contains no non-trivial constant. The fact that $d=0$ follows from the theorem.

\begin{rem}
  The proof of the characterization of $\mathcal L(\Phi)$ spaces in terms of the structural identity \eqref{structure-ide} extends to  the setting of Pontryagin or Krein spaces.
  In that case, we may have $\ker R_a\not=\left\{0\right\}$, and poles may occur in $\mathbb C\setminus\mathbb T$.
  \end{rem}

        \begin{thm}
          There exists a uniquely defined unitary operator $U$ in $\mathcal L(\Phi)$ such that
          \begin{equation}
            (U-aI_{\mathcal L(\Phi)})^{-1}=R_a,\quad a\in\mathbb C\setminus\mathbb T,
          \end{equation}
          and $U$ is given by
          \begin{equation}
            (UF)({\lambda})={\lambda}F({\lambda})+c_F,
          \end{equation}
          where
          \begin{equation}
c_F=-\lim_{w\rightarrow\infty}wF(w)=\int_0^{2\pi}d\mu(e^{it})f(e^{it})
\end{equation}
with $F$ as in \eqref{rep-F-phi}.
          \end{thm}

          \begin{proof}
Since the operators $R_a$ satisfy the resolvent identity and have a zero kernel, the existence of $U$ follows from Stone's theorem
            (see \cite[Theorem 4.10 p. 137]{Stone}) once we prove that they are bounded in $\mathcal L(\Phi)$.
            A direct and more precise proof can be given using the representation \eqref{rep-F-phi} for elements of $\mathcal L(\Phi)$. The formula for $c_F$ was given earlier in
            \eqref{lim+}.
           \end{proof}

           For the line case version of the following result, see \cite{dbr1} and \cite{MR2069002}.

            \begin{thm}
Assume that $\det\Phi\not\equiv 0$. Then the map $F\mapsto \Phi^{-1}F$ is unitary from $\mathcal L(\Phi)$ onto $\mathcal L(\Phi^{-1})$.
\end{thm}

\begin{proof}
The equalities
\[
  \begin{split}
    L_{\Phi}({\lambda},w)&=\Phi({\lambda}) L_{\Phi^{-1}}({\lambda},w)(\Phi(w))^*,\\
    L_{\Phi^{-1}}({\lambda},w)&=\Phi^{-1}({\lambda}) L_{\Phi}({\lambda},w)(\Phi^{-1}(w))^*,
    \end{split}
  \]
  show that the operator of multiplication by $\Phi^{-1}$ (resp. by $\Phi$) is an isometry from $\mathcal L(\Phi^{-1})$ into $\mathcal L(\Phi)$
  (resp. from $\mathcal L(\Phi)$ into $\mathcal L(\Phi^{-1})$). See \cite[Exercise 7.6.6 p. 363]{CAPB_2} for more information on multipliers in reproducing kernel
  Hilbert spaces.
   \end{proof}

        Finally we recall Zaremba's formula for the reproducing kernel $K({\lambda},w)$ of a reproducing kernel Hilbert space $\mathcal H$ of $\mathbb C^n$-valued functions
        defined on a set $\Omega$. If $(f_c)_{c\in C}$ is an orthonormal basis of the latter then
        \begin{equation}
K({\lambda},w)=\sum_{c\in C} f_c({\lambda})(f_c(w))^*,\quad {\lambda},w\in\Omega.
\label{zarema-form}
\end{equation}
We will apply this result to the space $\mathcal L(\Phi)$, which consist of analytic functions and hence is separable, and thus the index set $C$ in
\eqref{zarema-form} will be taken to be $\mathbb N$ there; see \eqref{zzrem}.

\section{The model of de Branges and Shulman}
\setcounter{equation}{0}
\label{2.1}
Here we present some of the links with \cite{dbs}. It should be noted that the present section is not a repetition of this latter source, but rather sheds new light on it.
In place of extending $\Phi$ to $\mathbb E$ via \eqref{2-2} de Branges and Shulman consider in \cite{dbs} the $\mathbb C^{2n\times 2n}$-valued kernel
\begin{equation}
  \label{ephi}
  E_\Phi({\lambda},w)=\begin{pmatrix}L_\Phi({\lambda},w)&\frac{\Phi({\lambda})-\Phi(\overline{w})}{2(\lambda-\overline{w})}\\
                \frac{\Phi^\sharp({\lambda})-\Phi^\sharp(\overline{w})}{2(\lambda-\overline{w})}&L_{\Phi^\sharp}({\lambda},w)\end{pmatrix},
            \end{equation}
where \begin{equation}
  \Phi^\sharp(\lambda)=(\Phi(\overline{\lambda}))^*
\end{equation}
and the associated reproducing kernel Hilbert space of $\mathbb C^{2n}$-valued functions analytic in $\mathbb D$.\\

With $\Phi$ with integral representation \eqref{phi-herg}, we set $g_\lambda$ to be the $\mathbb C^{n\times 2n}$-valued function
\begin{equation}
  g_{\lambda}(e^{it})=\begin{pmatrix} \frac{I_n}{e^{-it} -\overline{{\lambda}}}&\frac{e^{it}I_n}{e^{it}-\overline{{\lambda}}}\end{pmatrix},\quad \lambda\in\mathbb C\setminus\mathbb T.
\end{equation}

            \begin{thm}
 The kernel $E_\Phi$ is positive definite in the open unit disk and the associated reproducing kernel Hilbert space consists of the $\mathbb C^{2n}$-valued functions $F$ of the form
 \begin{equation}
\label{f-mu}
   \mathsf F(\lambda)=\int_0^{2\pi}\begin{pmatrix} \frac{1}{e^{it}-{\lambda}}I_n\\ \frac{e^{-it}}{e^{-it}-{\lambda}}I_n\end{pmatrix}d\mu(e^{it})f(e^{it}),
\end{equation}
with $f\in\mathbf L_2(d\mu)$ and norm
\[
  \|\mathsf F\|=\|f\|_\mu.
  \]
\end{thm}

\begin{proof}
For $c,d\in\mathbb C^{2n}$ we have
\begin{equation}
  \label{inner-prod}
  \langle g_wc,g_{\lambda}d\rangle_\mu=d^*E_\Phi({\lambda},w)c.
\end{equation}
The expression for the block diagonals of \eqref{ephi} follows from \eqref{ext-2-2} for $\lambda,w\in\mathbb D$ and  applied respectively to $\Phi$ and $\Phi^\sharp$. To compute
the upper right block of $E_\Phi$ we write for $u,v\in\mathbb C^n$,
  \[
    \begin{split}
      v^*\frac{\Phi(\lambda)-\Phi(\overline{w})}{2(\lambda-\overline{w})}u&=
                        \frac{1}{2(\lambda-\overline{w})}\int_0^{2\pi}v^*\left\{\frac{e^{it}+\lambda}{e^{it}-\lambda}-\frac{e^{it}+\overline{w}}{e^{it}-\overline{w}}\right\}
                                                                    u d\mu(e^{it})\\
                                                                   &=\int_0^{2\pi}v^*\left\{\frac{e^{it}}{(e^{it}-\lambda)(e^{it}-\overline{w})}\right\}ud\mu(e^{it})\\
      &=\langle \frac{e^{it}}{e^{it}-\overline{w}}u,\frac{1}{e^{-it}-\overline{\lambda}}v\rangle_\mu.
    \end{split}
  \]
  The lower left block is computed in the same way.
Formula \eqref{inner-prod} expresses $E_\phi$ in form of an inner product, from which follows the positive definiteness.
  \end{proof}

  \begin{thm}
    Let
    \[
      \mathsf F=\begin{pmatrix} \mathsf F_1\\ \mathsf F_2\end{pmatrix}
          \]
    be the decomposition of $F\in\mathcal H(E_\Phi)$ into two $\mathbb C^n$-valued blocks. The operator
    \begin{equation}
      \label{unit}
      (U\mathsf F)(\lambda)=\begin{pmatrix}(R_0\mathsf F_1)(\lambda)\\ \lambda \mathsf F _2(\lambda)+\mathsf F_1(0)
           \end{pmatrix}
    \end{equation}
      is unitary from $\mathcal H(E_\Phi)$ onto itself and corresponds to the multiplication by $e^{-it}$ in $\mathbf L_2^n(d\mu)$.    \end{thm}
    \begin{proof}
      It suffices to remark that \eqref{unit} corresponds to the function $e^{-it}f(e^{it})$ in $\mathbf L_2(d\mu)$ in \eqref{f-mu}. More precisely we have for $\lambda\not=0$
      \[
        \begin{split}
          (R_0\mathsf F_1)(\lambda)&=\frac{1}{\lambda}\int_0^{2\pi}\left\{\frac{1}{e^{it}-{\lambda}}-\frac{1}{e^{it}}\right\}d\mu(e^{it})f(e^{it})\\
                                   &=\int_0^{2\pi}\left\{\frac{1}{e^{it}-{\lambda}}\frac{1}{e^{it}}\right\}d\mu(e^{it})f(e^{it})\\
                    &=\int_0^{2\pi}\frac{1}{e^{it}-{\lambda}}d\mu(e^{it})e^{-it}f(e^{it})
        \end{split}
      \]
      and
      \[
        \begin{split}
          \lambda\mathsf F_2(\lambda)+\mathsf F_1(0)&=\int_0^{2\pi}\left\{\frac{\lambda e^{-it}}{e^{-it}-\lambda}+\frac{1}{e^{it}}\right\}d\mu(e^{it})f(e^{it})\\
                                                   &=\int_0^{2\pi}\frac{e^{-it}}{(e^{-it}-\lambda)e^{it}}d\mu(e^{it})f(e^{it})\\
                    &=\int_0^{2\pi}\frac{e^{-it}}{(e^{-it}-\lambda)}d\mu(e^{it})e^{-it}f(e^{it}).
        \end{split}
        \]
      \end{proof}
\section{The model for pairs of unitary operators}
\setcounter{equation}{0}
\label{sec-3}

The proof of the following theorem from \cite{MR1638044} is sketchy there and we have chosen to present a more detailed proof of the result.
\begin{thm}
  Let $U_+$ and  $U_-$ be two unitary operators in a Hilbert space $\mathcal H$ such that
  \begin{equation}
\label{ze-inter}
    \bigcap_{\lambda\in\mathbb C\setminus\mathbb T}\ker\left((U_+-\lambda I_{\mathcal H})^{-1}-(U_--\lambda I_{\mathcal H})^{-1}\right)=\left\{0\right\}
  \end{equation}
  and for which there exists $n\in\mathbb N$ such that
  \begin{equation}
    \label{rtg}
    {\rm dim}\,{\rm ran}\,\left((U_+-\lambda I_{\mathcal H})^{-1}-(U_--\lambda I_{\mathcal H})^{-1}\right)=n,\quad\forall \lambda\in\mathbb C\setminus\mathbb T.
  \end{equation}
 Then there exist a $\mathbb C^{n\times n}$-valued Carath\'eodory function $\Phi$ and maps
  \[
    W_+\,\,:\mathcal H\,\,\,\longrightarrow\,\,\mathcal L(\Phi)\quad and\quad
        W_-\,\,:\mathcal H\,\,\,\longrightarrow\,\,\mathcal L(\Phi^{-1})
      \]
      such that the diagram
\[
\begin{tikzcd}
\mathcal H \arrow[r, "W_+"] \arrow[d, "Id"'] & \mathcal L(\Phi) \\
\mathcal H \arrow[r, "W_-"]                  & \mathcal L(\Phi^{-1})\arrow[u,"M_{\Phi}" ']
\end{tikzcd}
\]
is commutative,
  \begin{equation}
    \label{comm}
W_+=M_\Phi W_-,
\end{equation}
and with the following properties, where $z\in\mathbb C\setminus\mathbb T$:
\begin{eqnarray}
  \label{w12}
  W_+(U_+-zI_{\mathcal H})^{-1}W_+^{-1}&=&R_z\quad {in}\quad\mathcal L(\Phi),\\
  W_-(U_--zI_{\mathcal H})^{-1}W_-^{-1}&=&R_z\quad {in}\quad\mathcal L(\Phi^{-1}).
\label{w22}
\end{eqnarray}
\end{thm}

\begin{proof} We proceed in a number of steps. For the first step, besides the proof in \cite{MR1638044} we refer when $n=1$ to the 1962 paper of Krein \cite{krein-trace2},
  also appearing in the collected works \cite[p. 342]{MR1710394}. See also the work \cite{MR3861897} and the citations therein.\\

  STEP 1: {\sl One can write
  \begin{equation}
    I_{\mathcal H}-U_+^*U_-=-C^*BC,
\label{law-1}
  \end{equation}
  where $C$ is a continuous operator from $\mathcal H$ into $\mathbb C^n$ and $B\in\mathbb C^{n\times n}$ is unitary such that}

  \begin{equation}
    \label{bbcc}
    B^{-1}+B^{-*}=-CC^*.
    \end{equation}

    We have
    \[
      (I_{\mathcal H}-U_+^*U_-)  (I_{\mathcal H}-U_+^*U_-)^*=2I_{\mathcal H}-(U_+^*U_-+U_-^*U_+),
    \]
    and so the operator $I_{\mathcal H}-U_+^*U_-$ is normal. Since
    \[
I_{\mathcal H}-U_+^*U_-=(U_-^*-U_+^*)U_-
\]
setting $\lambda=0$ in \eqref{rtg} shows  $I_{\mathcal H}-U_+^*U_-$  has finite rank. The spectral theorem allows us to write
    \[
I_{\mathcal H}-U_+^*U_-=\sum_{j=1}^N \frac{\lambda_j}{|\lambda_j|}|\lambda_j|E_j,
    \]
    where $E_1,\ldots, E_N$ are finite rank orthogonal projections, say of rank $d_1,\ldots, d_N$, and $\lambda_1,\ldots, \lambda_N$ are complex numbers different from
    $0$. Writing $E_j=C_j^*C_j$, where $C_j$ is from $\mathcal H$ into $\mathbb C^{d_j}$ we obtain the result with
\[
C=\begin{pmatrix}\sqrt{|\lambda_1|}C_1&\cdots &\sqrt{|\lambda_N|}C_N\end{pmatrix}\quad{\rm and}\quad
B=-{\rm diag}\, (\frac{\lambda_1}{|\lambda_1|}I_{d_1},\ldots, \frac{\lambda_N}{|\lambda_N|}I_{d_N}).
  \]
Note that taking adjoint of \eqref{law-1} we have

  \begin{equation}
    I_{\mathcal H}-U_-^*U_+=-C^*B^*C .
\label{law-2}
  \end{equation}

STEP 2: {\sl Set
  \begin{equation}
    \label{phi-model}
    \Phi(\lambda)=B^{-1}+C(I_{\mathcal H}-\lambda U_+^*)^{-1}C^*.
  \end{equation}
  Then}
    \begin{equation}
      \Phi^{-1}(\lambda)=B+BC(I_{\mathcal H}-\lambda U_-^*)^{-1}C^*B^*.
\label{inv-inv}
    \end{equation}

    We compute
    \[
      \begin{split}
        \Phi(\lambda)(B+BC(I_{\mathcal H}-\lambda U_-^*)^{-1}C^*B^*)&=\\
        &\hspace{-5cm}=(B^{-1}+C(I_{\mathcal H}-\lambda U_+^*)^{-1}C^*)(       B+BC(I_{\mathcal H}-\lambda U_-^*)^{-1}C^*B^*)\\
        &\hspace{-5cm}=I_n+C(I_{\mathcal H}-\lambda U_+^*)^{-1}C^*B+\\
        &\hspace{-4.5cm}+C(I_{\mathcal H}-\lambda U_-^*)^{-1}C^*B^*+\\
        &\hspace{-4.5cm}+C(I_{\mathcal H}-\lambda U_+^*)^{-1}C^*BC(I_{\mathcal H}-\lambda U_-^*)^{-1}C^*B^*\\
        &\hspace{-5cm}=I_n+C(I_{\mathcal H}-\lambda U_+^*)^{-1}\left\{C^*B+(I_{\mathcal H}-\lambda U_+^*)(I_{\mathcal H}-\lambda U_-^*)^{-1}C^*B^*+\right.\\
        &\hspace{-4.5cm}\left.+C^*BC(I_{\mathcal H}-\lambda U_-^*)^{-1}C^*B^*\right\}\\
        &\hspace{-5cm}=I_n+C(I_{\mathcal H}-\lambda U_+^*)^{-1}\times\\
        &\hspace{-4.5cm}\times\left\{C^*B+(I_{\mathcal H}-\lambda U_+^*+C^*BC)(I_{\mathcal H}-\lambda U_-^*)^{-1}C^*B^*\right\}.
            \end{split}
          \]
          But
          \[
            \begin{split}
             C^*B+ (I_{\mathcal H}-\lambda U_+^*+C^*BC)(I_{\mathcal H}-\lambda U_-^*)^{-1}C^*B^*&=\\
              &\hspace{-9cm}       =C^*B+      (U_+^*U_--\lambda U_+^*)(I_{\mathcal H}-\lambda U_-^*)^{-1}C^*B^*\\
              &\hspace{-9cm}=C^*B+       U_+^*(U_--\lambda I_{\mathcal H})(I_{\mathcal H}-\lambda U_-^*)^{-1}C^*B^*\\
              &\hspace{-9cm}=C^*B+    U_+^*U_-C^*B^*\\
              &\hspace{-9cm}=C^*B+      (I_{\mathcal H}+C^*BC)C^*B^*\\
              &\hspace{-9cm}=C^*\underbrace{\left(B+B^* +BCC^*B^*\right)}_{=0\,\,\,{\rm by}\,\, \eqref{bbcc}}\\
              &\hspace{-9cm}=0.
            \end{split}
            \]

STEP 3: {\sl We have the reproducing kernel formulas

  \begin{eqnarray}
    \label{eq-0001}
    \hspace{5mm}\frac{\Phi(\lambda)+(\Phi(w))^*}{2(1-\lambda\overline{w})}&=&\frac{1}{2}
                                                                              C(I_{\mathcal H}-\lambda U_+^*)^{-1}(I_{\mathcal H}-wU_+^*)^{-*}C^*,\\
    \hspace{5mm} \frac{\Phi^{-1}(\lambda)+(\Phi^{-1}(w))^*}{2(1-\lambda\overline{w})}&=&\frac{1}{2}
                                                                                         BC(I_{\mathcal H}-\lambda U_-^*)^{-1}(I_{\mathcal H}-w U_-^*)^{-*}C^*B^*
                                                                                        \nonumber
  \end{eqnarray}
  and in particular it satisfies}\eqref{2-2}.\smallskip

The identity \eqref{bbcc} plays a key role in the computation. We have
\[
  \begin{split}
    \Phi(\lambda)+(\Phi(w))^*&=B^{-1}+C(I_{\mathcal H}-\lambda U_+^*)^{-1}C^*+B^{-*}+C(I_{\mathcal H}-w U_+^*)^{-*}C^*\\
    &=-CC^*+C(I_{\mathcal H}-\lambda U_+^*)^{-1}C^*+C(I_{\mathcal H}-w U_+^*)^{-*}C^*\\
    &=C(I_{\mathcal H}-\lambda U_+^*)^{-1}\left\{-(I_{\mathcal H}-\lambda U_+^*)(I_{\mathcal H}-w U_+^*)^{-*}+\right.\\
    &\left.\hspace{5mm}+(I_{\mathcal H}-\lambda U_+^*)      +(I_{\mathcal H}-w U_+^*)^*
    \right\}(I_{\mathcal H}- w U_+^*)^{-*}C^*\\
    &=(1-\lambda\overline{w})C(I_{\mathcal H}-\lambda U_+^*)^{-1}(I_{\mathcal H}- w U_+^*)^{-*}C^*.
  \end{split}
  \]
  The proof for $\Phi^{-1}$ is similar and uses formula \eqref{inv-inv}.\\

STEP 4: {\sl The space $\mathcal L(\Phi)$ consists of the functions of the form
\[
  F(\lambda)=C(I_{\mathcal H}-\lambda U_+^*)^{-1}f,
\]
  with $f\in\mathcal H$ and norm $\|F\|=\|f\|_{\mathcal H}$.}\smallskip

From \eqref{eq-0001} it is enough to check that
\[
  C(I_{\mathcal H}-\lambda U_+^*)^{-1}f\equiv 0\quad\Longrightarrow\quad f=0.
  \]
Let therefore $f$ satisfying this condition. We have then
  \[
(U_--\lambda I_{\mathcal H})^{-1}U_-C^*B^*C(U_+-\lambda I_{\mathcal H})^{-1}U_+f\equiv 0.
    \]
Using \eqref{law-2} we rewrite this equality as
  \[
(U_--\lambda I_{\mathcal H})^{-1}(U_--U_+)(U_+-\lambda I_{\mathcal H})^{-1}U_+f\equiv 0.
    \]
    Writing
    $$U_--U_+=U_--\lambda I_{\mathcal H}+\lambda I_{\mathcal H}-U_+$$
    we see  from \eqref{ze-inter} that $U_+f=0$, and so $f=0$.\\

    It follows from the previous step that the formulas
  \begin{eqnarray}
    (W_+f)(\lambda)&=&\frac{1}{\sqrt{2}}C(I_{\mathcal H}-\lambda U_+^*)^{-1}U_+^*f,\\
    (W_-f)(\lambda)&=&\frac{1}{\sqrt{2}}C(I_{\mathcal H}-\lambda U_-^*)^{-1}U_-^*f
  \end{eqnarray}
  define unitary maps from $\mathcal H$ onto $\mathcal L(\Phi)$ and $\mathcal L(\Phi^{-1})$, respectively.\\

  STEP 4: {\sl \eqref{w12}-\eqref{w22} hold.}\smallskip

  We prove \eqref{w12}. The proof of \eqref{w22} is similar. Let $f\in\mathcal H$. We write
  \[
    \begin{split}
     ( W_+(U_+-z I_{\mathcal H})^{-1}f)(\lambda)&=C(I_{\mathcal H}-\lambda U_+^*)^{-1}U_+^*(U_+-z I_{\mathcal H})^{-1}f\\
      &=C(U_+-\lambda I_{\mathcal H})^{-1}(U_+-z I_{\mathcal H})^{-1}f\\
      &=C\left\{\frac{(U_+-\lambda I_{\mathcal H})^{-1}-(U_+-z I_{\mathcal H})^{-1}}{\lambda-z}\right\}f\\
      &=C\left\{\frac{(I_{\mathcal H}-\lambda U_+^*)^{-1}-(I_{\mathcal H}-zU_+^*)^{-1}}{\lambda-z}\right\}U_+^*f\\
      &=(R_zW_+f)(\lambda).
    \end{split}
    \]

  STEP 5: {\sl We prove \eqref{comm}:}\smallskip

  Let $f\in\mathcal H$. We have:

  \[
    \begin{split}
      \Phi(\lambda)(W_-f)(\lambda)&=(B^{-1}+C(I_{\mathcal H}-\lambda U_+^*)^{-1}C^*)BC(I_{\mathcal H}-\lambda U_-^*)^{-1}U_-^*f\\
      &=C(I_{\mathcal H}-\lambda U_-^*)^{-1}U_-^*f+\\
      &\hspace{5mm}+C(I_{\mathcal H}-\lambda U_+^*)^{-1}\overbrace{C^*BC}^{U_+^*U_--I_{\mathcal H}}(I_{\mathcal H}-\lambda U_-^*)^{-1}U_-^*f\\
      &=C(I_{\mathcal H}-\lambda U_+^*)^{-1}\left\{I_{\mathcal H}-\lambda U_+^*+U_+^*U_--I_{\mathcal H}\right\}\times\\
      &\hspace{5mm}\times(I_{\mathcal H}-\lambda U_-^*)^{-1}U_-^*f\\
      &=C(I_{\mathcal H}-\lambda U_+^*)^{-1}U_+^*\underbrace{\left\{U_--\lambda I_{\mathcal H}\right\}(I_{\mathcal H}-\lambda U_-^*)^{-1}U_-^*}_{I_{\mathcal H}}f\\
      &=(W_+f)(\lambda).
    \end{split}
  \]
\end{proof}

We now consider the inverse problem: Starting from $\Phi$ can we find $U_\pm$ satisfying \eqref{ze-inter} and \eqref{rtg}. The result is presented in Theorem
\ref{isp}. We first need a preliminary result and set
  \begin{equation}
    \Delta(z)=(U_+-zI_{\mathcal H})^{-1}-(U_--zI_{\mathcal H})^{-1}.
  \end{equation}

\begin{prop}
  In the above notation,
let $\lambda,z\in\mathbb C\setminus\mathbb T$. Then:
  \begin{equation}
    \label{delta-1}
(W_+\Delta(z)W_+^{-1}F)(\lambda)=(R_z\Phi(\lambda))\Phi^{-1}(z)F(z).
\end{equation}
  \label{prop11}
  \end{prop}
  \begin{proof}

To prove \eqref{delta-1} we take use of \eqref{comm} and write
\[
\begin{split}
  (W_+(U_--zI_{\mathcal H})^{-1}W_+^{-1}F)(\lambda)&=\Phi(\lambda)(\underbrace{( W_-(U_--zI_{\mathcal H})^{-1}W_-^{-1})}_{R_z}(\Phi^{-1}F))(\lambda)\\
  &\hspace{-2cm}=\Phi(\lambda)\frac{\Phi^{-1}(\lambda)F(\lambda)-\Phi^{-1}(z)F(z)}{\lambda-z}\\
  &\hspace{-2cm}=\frac{F(\lambda)-F(z)}{\lambda-z}-\frac{\Phi(\lambda)-\Phi(z)}{\lambda-z}\Phi^{-1}(z)F(z)\\
  &\hspace{-2cm}=(R_zF)(\lambda)-(R_z\Phi(\lambda))\Phi^{-1}(z)F(z),\\
  &\hspace{-2cm}=(W_+(U_+-zI_{\mathcal H})^{-1}W_+^{-1}F)(\lambda)-(R_z\Phi(\lambda))\Phi^{-1}(z)F(z),\\
\end{split}
\]
and the result follows.
    \end{proof}

We can now state and prove:

\begin{thm}
  \label{isp}
  Let $\Phi$ be a $\mathbb C^{n\times n}$-valued Carath\'eodory function such that
\[
  \det \mbox{{\rm Re}}\,\Phi({\lambda})\not\equiv0
\]
  in $\mathbb D$ and extended to $\mathbb E$ by
  the anti-symmetry condition \eqref{2-2}. Then the operators $U_+$ and $U_-$ defined by the resolvent operators $R_z$ in $\mathcal L(\Phi)$ and $\mathcal L(\Phi^{-1})$
  satisfy \eqref{ze-inter} and \eqref{rtg}.
\end{thm}

\begin{proof}
  Now $W_+$ is the identity in the arguments above, and  \eqref{delta-1} shows that the rank of $\Delta(z)$ is less or equal to $n$. If it was strictly less than $n$
  we would have that the span of the vectors $F(\lambda_0)$ for a fixed $\lambda_0\in\mathbb C\setminus\mathbb T$ and $F$ running through $\mathcal L(\phi)$
  would be strictly included inside $\mathbb C^n$.
  Then for some $\xi\in\mathbb C^n$ we have
  \[
\xi^*\frac{\Phi(\lambda_0)+(\Phi(a))^*}{1-\lambda_0\overline{a}}c=0,\quad \forall a\in \mathbb C\setminus\mathbb T\quad{\rm and}\quad \forall c\in\mathbb C^n.
\]
Setting $a=\lambda_0$ and $c=\xi$ we get that ${\rm Re}\, \Phi(\lambda_0)$ is not invertible.\smallskip

  That same formula shows that $F$ in the intersection \eqref{ze-inter} will satisfy $F(\lambda)=0$ for $\lambda\in\mathbb C\setminus\mathbb T$.
  \end{proof}

\section{The trace formula}
\setcounter{equation}{0}
\label{Tr1}
\begin{thm}
  Let $U_+$ and  $U_-$ be two unitary operators in a Hilbert space $\mathcal H$ satisfying \eqref{ze-inter} and \eqref{rtg}, and let $\Phi$ be the associated Carath\'eodory
  function. Then, it holds that
  \begin{eqnarray}
    \label{tre-delta}
  (\det B^{-1})(\det(U_+-{{\lambda}} I_{\mathcal H})^{-1}(U_--{{\lambda}} I_{\mathcal H}))&=&   \det\Phi({\lambda}),\\
    {\rm Tr}\,\left((U_+-{\lambda}I_{\mathcal H})^{-1}-(U_--{\lambda}I_{\mathcal H})^{-1}\right)&=&{\rm Tr}\, \Phi^{-1}({\lambda})\Phi^\prime({\lambda}).
                                                                                       \label{tre-delta-2}
    \end{eqnarray}
  \end{thm}

\begin{proof}

From \eqref{phi-model} we have where we used \eqref{law-1} to go from the second to the third line
  \[
    \begin{split}
\det\Phi({\lambda})&=(\det B^{-1})\det(I_n+BC(I_{\mathcal H}-{{\lambda}} U_+^*)^{-1}C^*)\\
&=(\det B^{-1})\det(I_{\mathcal H}+(I_{\mathcal H}-{{\lambda}} U_+^*)^{-1}C^*BC)\\
&=(\det B^{-1})\det(I_{\mathcal H}-(I_{\mathcal H}-{{\lambda}} U_+^*)^{-1}(I_{\mathcal H}-U_+^*U_-))\\
&=(\det B^{-1})\det(I_{\mathcal H}-(U_+-{{\lambda}} I_{\mathcal H})^{-1}(U_+-U_-))\\
&=(\det B^{-1})(\det(U_+-{{\lambda}} I_{\mathcal H})^{-1}(U_--{{\lambda}} I_{\mathcal H})).
\end{split}
\]
To obtain \eqref{tre-delta-2} we differentiate both sides of  \eqref{tre-delta} using  two formulas pertaining to derivative of a determinant.  First
using the formula (see \cite[Chapter 4, Section 2 p. 132]{MR86m:00014}) we obtain
        \begin{equation}
          \label{log-det}
          \frac{d}{d\lambda} \ln\left(    \det (B-{\lambda} I_{\mathcal H})(A-{\lambda} I_{\mathcal H})^{-1}\right)
          ={\rm Tr}\, \left\{(A-{\lambda}I_{\mathcal H})^{-1}-(B-{\lambda}I_{\mathcal H})^{-1}\right\},
          \end{equation}
          where $A$ and $B$ are possibly unbounded operator such that $A-B$ is trace class we have
          \begin{equation}
            \label{eqw1}
\frac{d}{d\lambda}\ln\det(U_--{\lambda}I_{\mathcal H})(U_+-{\lambda} I_{\mathcal H})^{-1}=\Delta({\lambda}),
\end{equation}
(where $\Delta({\lambda})$ denotes the left hand side of \eqref{tre-delta-2}). Next  we obtain
\[
\frac{d}{d{\lambda}}\ln\det\Phi({\lambda})={\rm Tr}\, \Phi^{-1}({\lambda})\Phi^\prime({\lambda})
  \]
from \cite[p. 129]{MR86m:00014}.
\end{proof}

   Assume that $\det\Phi\not=0$. in both $\mathcal L(\Phi)$.  Then the model we give relates the ``spectra'' associated to the measures
        $d\mu_\pm$ associated to $\Phi$ and $\Phi^{-1}$ in Herglotz' representation. Furthermore, by the functional calculus we have:

      \begin{thm}
Let $R>1$. Then, we have
        \begin{equation}
          {\rm Tr}\,\,\left\{ f(U_-)-f(U_+)\right\}=\frac{1}{2\pi i}\int_{|z|=R}f(z){\rm Tr}\, \Phi^{-1}(z)\Phi^\prime(z)dz .
        \end{equation}
      \end{thm}

      Note that $\Phi$ is not in general meromorphic across the unit circle and one cannot compute the integral above using residue theory.
      When $\Phi$ is rational and does satisfy \eqref{2-2} see Theorem \ref{ththth} for a special case of the above result.

\section{Another proof of the trace formula}
\setcounter{equation}{0}
\label{tr2}
We now present another proof of the trace formula. It is admittedly longer, but using Zaremba's formula provides, in our opinion, supplementary insight to the result.
We divide the argument into a number of steps.\\

STEP 1: {\sl For $z\not=0$ it holds that}
\begin{equation}
  (R_z\Phi)(\lambda)=-\frac{1}{z}K_\Phi\left(\lambda,\frac{1}{\overline{z}}\right),
  \label{rzphi}
\end{equation}
and in particular
\begin{equation}
K_\Phi\left(z,\frac{1}{\overline{z}}\right)=-z\Phi^{\prime}(z).
  \label{der}
  \end{equation}
Using \eqref{2-2} we can write
\[
  (R_z\Phi)(\lambda)=\frac{\Phi(\lambda)-\Phi(z)}{\lambda-z}
  =\frac{\Phi(\lambda)+\left(\Phi\left(\frac{1}{\overline{z}}\right)\right)^*}{-{z}\left(1-\lambda{\overline{\frac{1}{\overline{z}}}}\right)}
  \]
which shows \eqref{rzphi}, from which \eqref{der} follows directly.\\

STEP 2: {\sl Let $f\in\mathcal H$ and $F=W_+f$. Then,}
\begin{equation}
  \label{345-543}
\langle \Delta(z)f,f\rangle_{\mathcal H}=-\frac{2}{z}{\rm Tr}\,\left(\Phi(z)^{-1}F(z)\left(F\left(\frac{1}{\overline{z}}\right)\right)^*\right).
    \end{equation}

    Using Proposition \ref{prop11} and \eqref{rzphi} we have
    \[
      \begin{split}
        \langle \Delta(z)f,f\rangle_{\mathcal H}&=\langle(R_z\Phi)\Phi^{-1}(z)F(z),F\rangle_{\mathcal L(\Phi)}\\
        &=\langle -\frac{1}{z}K_\Phi\left(\cdot,\frac{1}{\overline{z}}\right)\Phi^{-1}(z)F(z),F\rangle_{\mathcal L(\Phi)}\\
        &=-\frac{1}{z}\overline{\langle F,K_\Phi\left(\cdot,\frac{1}{\overline{z}}\right)\Phi^{-1}(z)F(z)\rangle_{\mathcal L(\Phi)}}\\
        &=-\frac{2}{z}\overline{(F(z))^*(\Phi^{-1}(z))^*F\left(\frac{1}{\overline{z}}\right)}\\
        &=-\frac{2}{z}\left(F\left(\frac{1}{\overline{z}}\right)\right)^*\Phi^{-1}(z)F(z)\\
        &=-\frac{2}{z}{\rm Tr}\,\left(\Phi^{-1}(z)F(z)\left(F\left(\frac{1}{\overline{z}}\right)\right)^*\right),
      \end{split}
    \]
    which is \eqref{345-543}.\\

    STEP 3: {\sl One proves the trace formula using Zaremba's formula for the reproducing kernel.}\\

    Let $e_1,e_2,\ldots$ be an orthonormal basis of $\mathcal H$ and let $F_1=W_+e_1,F_2=W_+e_2,\ldots$ be the corresponding orthonormal basis in
    $\mathcal L(\Phi)$.     By the vector-valued version of Zaremba's formula we have
    \begin{equation}
      \label{zzrem}
L_\Phi(z,w)=\sum_{\ell=1}^\infty F_\ell(z)F_\ell(w)^*,
\end{equation}
and so for $w=\frac{1}{\overline{z}}$, using \eqref{der}, we have
\begin{equation}
  \label{qazxsw}
  L_\Phi(z,\frac{1}{\overline{z}})=\sum_{\ell=1}^\infty F_\ell(z)\left(F_\ell
    \left(\frac{1}{\overline{z}}\right)\right)^*=-\frac{z}{2}\Phi^{\prime}(z).
\end{equation}
By definition of the trace, and using \eqref{345-543},
    \[
      \begin{split}
        {\rm Tr}\, (\Delta(z))&=\sum_{\ell=1}^\infty \langle\Delta(z)e_\ell,e_\ell
                                \rangle_{\mathcal H}\\
            &=-\frac{1}{z}\sum_{\ell=1}^\infty {\rm Tr}\,\left(\Phi^{-1}(z)F_\ell(z)\left(F_\ell\left(\frac{1}{\overline{z}}\right)\right)^*\right)\\
      &=-\frac{1}{z}{\rm Tr}\,\left(\Phi^{-1}(z)\underbrace{\sum_{\ell=1}^\infty
          F_\ell(z)\left(F_\ell\left(\frac{1}{\overline{z}}\right)\right)^*}_{-z\Phi^\prime(z)\,\, (\mbox{{\rm see \eqref{qazxsw}}})}\right)\\
      &={\rm Tr}\, \Phi^{-1}(z)\Phi^\prime(z).
\end{split}
    \]

      \section{The rational case}
      \setcounter{equation}{0}
      \label{ra-case}
      We now consider the case where $\Phi$ is rational, and first recall that a $\mathbb C^{p\times q}$-valued rational function
      analytic at infinity can be written in the form
      \[
        \Phi({\lambda})=D+C({\lambda}I_N-A)^{-1}B,
      \]
      where $D=\Phi(\infty)$ and $A,B,C$ are matrices of appropriate sizes. The realization is said to be minimal when the size $N$ of the matrix $A$ is minimal. When the function
      is analytic at the origin the realization takes now the form
      \[
\Phi({\lambda})=D+{\lambda}C(I_N-{\lambda}A)^{-1}B.
\]
In that latter case, applying first the realization result to ${\lambda}\Phi({\lambda})=D+{\lambda}C(I_N-{\lambda}A)^{-1}B$ one has $D=0$ and gets the realization
\[
  \Phi({\lambda})=C(I_N-{\lambda}A)^{-1}B,
\]
for $\Phi({\lambda})$. This latter realization is not minimal, even if one starts from a minimal realization of ${\lambda}\Phi({\lambda})$. See e.g. \cite{MR2363355,MR2663312} for more on realization
of matrix-valued rational functions.\\

We begin with a general result, where the function $\Phi$ is not assumed to have a positive real part in $\mathbb D$.

            \begin{thm}
              Let $\Phi$ be a $\mathbb C^{n\times n}$-valued rational function, analytic at infinity, with $\Phi(\infty)$ invertible (but $\Phi$ need not satisfy \eqref{2-2}).
              Let $\Phi({\lambda})=D+C({\lambda}I_N-A)^{-1}B$ be a realization of $\Phi$ centered at $\infty$. Then,
\begin{eqnarray}
  \label{eqas-1}
  \det\Phi({\lambda})&=&(\det D^{-1})\det\left(({\lambda}I_N-A^\times)({\lambda}I_N-A)^{-1}\right)\\
  {\rm Tr}\,\Phi^\prime({\lambda})\Phi({\lambda})^{-1}&=&{\rm Tr}\left\{({\lambda}I_N-A^\times)^{-1}-({\lambda}I_N-A)^{-1}\right\},
                                          \label{eqas-2}
\end{eqnarray}
where $A^\times=A-BD^{-1}C$.
            \end{thm}

            \begin{proof}
To prove \eqref{eqas-1} we write
          \[
             \begin{split}
              \det\Phi({\lambda})&=(\det D)\det(I_p+D^{-1}{C}({\lambda}I_N-{A})^{-1}{B})\\
              &=(\det D)\det(I_N+({\lambda}I_N-{A})^{-1}{B}D^{-1} C)
              \end{split}
            \]
            from which  \eqref{eqas-1} follows using
\begin{equation}
                \label{tr-formula}
                BD^{-1}C=({\lambda}I_N-A^\times)-({\lambda}I_N-A).
                \end{equation}

         We now turn to \eqref{eqas-2}.     We have $\Phi^\prime({\lambda})=-C({\lambda}I_N-A)^{-2}B$ and so
              \[
                \begin{split}
                  \Phi^\prime({\lambda})\Phi^{-1}({\lambda})&=-C({\lambda}I_N-A)^{-2}B(D^{-1}-D^{-1}C({\lambda}I_N-A^\times)^{-1}BD^{-1})\\
                  &=-C({\lambda}I_N-A)^{-2}BD^{-1}+\\
                                                            &\hspace{5mm}+C({\lambda}I_N-A)^{-2}BD^{-1}C({\lambda}I_N-A^\times)^{-1}
                                                              BD^{-1},\\
                \end{split}
              \]
             so that

              \[
 \Phi^\prime({\lambda})\Phi^{-1}({\lambda})=-C({\lambda}I_N-A)^{-1}({\lambda}I_N-A^\times)^{-1}BD^{-1}.
\]
Taking trace on both sides,
\[
  {\rm Tr}\,                  \Phi^\prime({\lambda})\Phi^{-1}({\lambda})=
  {\rm Tr}\, \left\{({\lambda}I_N-A^\times)^{-1}(-BD^{-1}C)({\lambda}I_N-A)^{-1}\right\}
  \]
  and using once more \eqref{tr-formula} we obtain \eqref{eqas-2}.
  \end{proof}

In the present setting two cases need to be distinguished
      \begin{itemize}
      \item $\Phi$ does not satisfy \eqref{2-2}. Then the function $\Phi$ extended to $\mathbb C\setminus\mathbb T$ via \eqref{2-2}
        will not be meromorphic in the complex plane in general. See Example \ref{exexex} for an illustration. The KYP (Kalman-Yakubovich-Popov) lemma
        gives precise characterization of minimal realizations of Carath\'eodory functions. We will not recall the precise statement in the present work but send the
        reader to \cite{MR525380,kyp4,kyp3,kyp2,Rantez,kyp1}.\\

      \item $\Phi$ satisfies \eqref{2-2}. It is then
        meromorphic in the complex plane.
        For instance $\Phi(\lambda)=\frac{1+{\lambda}}{1-{\lambda}}$.
        Then the description of the minimal realizations is much easier and presented below (see Theorem \ref{real-0987} below).
\end{itemize}

        In the first case we have:

                \begin{thm}
          Let
          \[
\Phi^\prime(\lambda)\Phi^{-1}(\lambda)=C(I_N-\lambda A)^{-1}B
\]
be a realization of the rational function $\Phi^\prime(\lambda)\Phi^{-1}(\lambda)$. Then,
\begin{equation}
\label{lala}
    {\rm Tr}\,( U_+^{*(\ell+1)}-U_-^{*(\ell+1)})={\rm Tr}\, CA^\ell B,\quad \ell=0,1,\ldots .
\end{equation}
          \end{thm}
          \begin{proof}
            This follows by comparing coefficients of the power expansion at the origin since
            \[
              (U_+-\lambda I_{\mathcal H})^{-1}-(U_--\lambda I_{\mathcal H})^{-1}=\sum_{\ell=0}^\infty
              \lambda^{\ell}(U_+^{*(\ell+1)}-U_-^{*(\ell+1)}), \quad |\lambda|<1.
\]
            \end{proof}

In the latter case one has following result for rational matrix-valued functions which satisfy \eqref{2-2}. The proof for the case of functions taking self-adjoint values on
the unit circle can be found in \cite[\S 5]{ag}. We note that
\eqref{45} means that the matrix $\mathscr A$ is unitary in the metric defined by $H^{-1}$, or equivalently, $\mathscr A$ is similar to a unitary matrix.

      \begin{thm}
        Let $\Phi$ a $\mathbb C^{n\times n}$-valued rational function analytic in neighborhoods of the origin and infinity, and let
        \[
          \Phi({\lambda})={\mathscr D}+{\mathscr C}({\lambda}I_N-{\mathscr A})^{-1}{\mathscr B}
        \]
        be a minimal realization of $\Phi$. Then $\Phi$
        satisfies \eqref{2-2} if and only if there exists an invertible Hermitian $H\in\mathbb C^{N\times N}$ such that
        \begin{eqnarray}
          \label{45}
          {\mathscr A}H^{-1}{\mathscr A}^*&=&H^{-1}\\
          {\mathscr B}&=&-H^{-1}{\mathscr A}^{-*}{\mathscr C^*}\\
          {\mathscr D}+{\mathscr D}^*&=&-{\mathscr C}H^{-1}{\mathscr C^*}.
                                         \label{47}
        \end{eqnarray}
        Then,
        \begin{equation}
          \label{59}
          \Phi({\lambda})=X+\frac{1}{2}{\mathscr C}(\mathscr A-{\lambda}I_N)^{-1}(\mathscr A+{\lambda}I_N)H^{-1}{\mathscr C^*},\quad X+X^*=0,
        \end{equation}
        and
        \begin{equation}
          \frac{\Phi({\lambda})+(\Phi(w))^*}{2(1-{\lambda}\overline{w})}={\mathscr C}({\lambda}I_N-{\mathscr A})^{-1}H^{-1}(wI_N-{\mathscr A})^{-*}{\mathscr C^*}.
          \label{58}
        \end{equation}
        \label{real-0987}
        \end{thm}

        \begin{proof}
          We first remark that $\mathscr A$ is invertible since $\Phi$ is assumed without a pole at the origin. Thus a minimal realization of $-(\Phi(1/\overline{{\lambda}}))^*$ is given by
          \begin{equation}
            -(\Phi(1/\overline{{\lambda}}))^*=-({\mathscr D}^*-{\mathscr B}^*{\mathscr A}^{-*}{\mathscr C}^*)+{\mathscr B}^*
              {\mathscr A}^{-*}({\lambda}I_N-{\mathscr A}^{-*})^{-1}{\mathscr A}^{-*}{\mathscr C^*}.
\end{equation}
Writing that two minimal realizations of a given rational functions are similar we see that there exists an invertible uniquely defined similarity matrix, which we
denote by $-H^{-1}$, and such that
\[
  \begin{pmatrix}{\mathscr A}&{\mathscr B}\\{\mathscr C}&{\mathscr D}\end{pmatrix}\begin{pmatrix}-H^{-1}&0\\0&I_n\end{pmatrix}=
  \begin{pmatrix}-H^{-1}&0\\0&I_n\end{pmatrix}\begin{pmatrix}{\mathscr A}^{-*}&-{\mathscr A}^{-*}\mathscr C^*\\
    {\mathscr B}^*{\mathscr A}^{-*}&-({\mathscr D}^*-{\mathscr B}^*{\mathscr A}^{-*}{\mathscr C^*})\end{pmatrix}.
\]
Equating the block entries in the above leads to \eqref{45}-\eqref{47}. One checks that $H^*$ also satisfies these equations and so $H=H^*$ by uniqueness of the similarity
matrix relating two minimal realizations. Furthermore, \eqref{59} and \eqref{58} are readily verified and $H>0$ follows from \eqref{58}.
          \end{proof}

 Note that \eqref{47} implies that $\mathscr D$ has a non-negative real part, but $\mathscr D$ need not be invertible. We also note that \eqref{45}-\eqref{47} can be rewritten
 as
 \[
   \begin{pmatrix}\mathscr{A}H^{-1}\mathscr{A}^*&\mathscr{A}H^{-1}\mathscr{C}^*\\
     \mathscr{ C}H^{-1}\mathscr{A}^*&\mathscr{C}H^{-1}\mathscr{C}^*\end{pmatrix}=\begin{pmatrix}H^{-1}&-\mathscr{B}\\ -\mathscr{B}^*&-\mathscr{D}-
     \mathscr{D}^*\end{pmatrix}.
 \]
 \begin{rem}
In the present setting we have $H>0$ since the kernel is positive. The argument uses the fact that the pair $(\mathscr C,\mathscr A)$ is observable.
\end{rem}

 Furthermore, since $\mathscr A$ is similar to the unitary matrix $U=H^{1/2}\mathscr {A}H^{-1/2}$ we can rewrite \eqref{59} as
 \begin{equation}
\Phi(\lambda)=-\mathscr C_1\left(\frac{1}{2}I_n+(\lambda U^*-I_N)^{-1}\right)\mathscr C^*_1+iX_1,
\end{equation}
where $X_1=X_1^*\in\mathbb C^{n\times n}$ and $\mathscr C_1=\mathscr C H^{1/2}$. We leave the verification of these computations to the reader.\\

Another remark is that the equivalence of this result with \eqref{link} is seen by rewritting the latter (in the notation of \eqref{link}) as
\[
  \begin{split}
    \Phi(\lambda)&=iE+\sum_{\ell=1}^mM_\ell \frac{e^{it_\ell}-\lambda+2\lambda}{e^{it_\ell}-\lambda}\\
    &=iE+\sum_{\ell=1}^m M_\ell
      +2\lambda\sum_{\ell=1}^m\frac{M_\ell e^{-it_\ell}}{1-\lambda e^{-it_\ell}},
    \end{split}
  \]
  corresponding to
  \[
    \begin{split}
      \mathscr A&={\rm diag}\,(e^{-it_1},\ldots , e^{-it_m})\\
      \mathscr B&=\sqrt{2}\begin{pmatrix}e^{-it_1}\sqrt{M_1}\\ \vdots \\ e^{-it_m}\sqrt{M_m}\end{pmatrix}\\
      \mathscr C&=\sqrt{2}\begin{pmatrix}\sqrt{M_1}& \cdots &\sqrt{M_m}\end{pmatrix}\\
      \mathscr D&=iE+\sum_{\ell=1}^m M_\ell,
                  \end{split}\]
      which satisfy \eqref{45}-\eqref{47} with $H=I_m$.\\

      Finally, to close this discussion on Theorem \ref{real-0987} we note that $\Phi(\lambda^N)$ fits the hypothesis of the theorem for any $N\in\mathbb N$ if $\Phi$ does. For
      $\Phi(\lambda)=\frac{1+\lambda}{1-\lambda}$ we have
      \[
\Phi(\lambda^N)=\frac{1+\lambda^N}{1-\lambda^N},
\]
corresponding to
  \[
    \begin{split}
      \mathscr A&=\begin{pmatrix} 0&I_{N-1}\\1&0\end{pmatrix},\\
      \mathscr B&=\sqrt{2}e_N,\\
      \mathscr C&=-\sqrt{2}e_1^*,\\
      \mathscr D&=-1,
    \end{split}\]
  where $e_1$ and $e_N$ are the standard unit vectors in the canonical basis of $\mathbb C^N$, which satisfy \eqref{45}-\eqref{47} with $H=I_m$.\\

We will need the following result in the scalar case. For our present purpose we do not need the definitions of zeros and poles of matrix-valued rational functions.
See e.g. \cite{kaash-field} for an introduction to the latter.

\begin{cor}
  Let $\Phi$ be a $\mathbb C^{n\times n}$-valued rational Carath\'eodory function satisfying the symmetry \eqref{2-2}. Then the poles and zeros of $\Phi$ and $\Phi^{-1}$ are on the
  unit circle.
\end{cor}

\begin{proof}
In the representation of both $\Phi$ and $\Phi^{-1}$ the measure is a jump measure with a finite number of jumps.
  \end{proof}

Recalling \eqref{rtg}, when $n=1$ we have the following result; recall that since $\Phi$ is rational and satisfies the symmetry \eqref{2-2} its poles or zeros are on the unit circle.
        \begin{thm}
          \label{ththth}
          Assume $n=1$ and $\Phi$ rational and satisfies \eqref{2-2}. Then,
        \begin{equation}
          {\rm Tr}\,\,\left\{ f(U_-)-f(U_+)\right\}=\sum_{\lambda \,\,\mbox{\rm a pole or zero of $\Phi$}} f(\lambda).
        \end{equation}
      \end{thm}

      \begin{proof} We write
        \begin{equation}
          f(\lambda)=\sum_{\ell=0}^\infty f_\ell\lambda^\ell
        \end{equation}
        and divide the proof into a number of steps.\\

        STEP 1: {\sl It holds that
\[
  ( \lambda I_{\mathcal H}-U_-)^{-1}-({\lambda}I_{\mathcal H}-U_+)^{-1}=\sum_{\ell=0}^\infty
  \frac{1}{\lambda^{\ell+1}}\left(U_-^\ell-U_+^\ell\right),
\]
where the convergence is in norm for $|\lambda|>1$.\\
}

The next step is an improvement of \eqref{lala}, due to the meromorphicity of $\Phi$.\\

STEP 2: {\sl There exists a constant $c$ such that

\begin{equation}
  | {\rm Tr}\,(U_-^{\ell+1}-U_+^{\ell+1})|\le c ,\quad \ell=0,1,\ldots
  \label{lalala}
\end{equation}
}\\

Indeed, from \eqref{eqas-2} we have:\\

STEP 3: {\sl With  $R>1$ one has:
  \begin{equation}
    \frac{1}{2\pi i}\int_{|\lambda|=R}f(\lambda)\left(( \lambda I_{\mathcal H}-U_-)^{-1}-({\lambda}I_{\mathcal H}-U_+)^{-1}\right)d\lambda=f(U_-)-f(U_+).
    \label{trace-987}
  \end{equation}
}
Let $\e>0$. From \eqref{lalala} there exists $N$ such that
\begin{equation}
  \label{7-16}
\sum_{\ell=0}^\infty\frac{|{\rm Tr}\, (U_-^\ell-U_+^\ell)|}{R^{\ell+1}}<\e.
  \end{equation}
  We write
  \[
    \begin{split}
      \frac{1}{2\pi i}\int_{|\lambda|=R}f(\lambda)\left({\rm Tr}\,\left(( \lambda I_{\mathcal H}-U_-)^{-1}-({\lambda}I_{\mathcal H}-U_+)^{-1}\right)\right)d\lambda&=\\
      &\hspace{-6cm}=
        \frac{1}{2\pi i}\int_{|\lambda|=R}f(\lambda)\left({\rm Tr}\,
        \left(\sum_{\ell=0}^\infty \frac{1}{\lambda^{\ell+1}}\left(U_-^\ell
        -U_+^\ell\right)\right)\right)d\lambda\\                                                          &\hspace{-6cm}
                                                                                                            =\frac{1}{2\pi i}\int_{|\lambda|=R}\left({\rm Tr}\,\left(\sum_{\ell=0}^N
                                                                                                            \frac{f(\lambda)}{\lambda^{\ell+1}}\left(U_-^\ell-U_+^\ell
                                                                                                            \right)\right)\right.\\
      &\hspace{-3.5cm}\left.
        +\sum_{\ell=N+1}^\infty \left(\frac{f(\lambda)}{\lambda^{\ell+1}}\left(U_-^\ell
        -U_+^\ell\right)\right)\right)d\lambda\\
    &\hspace{-6cm}=\spadesuit+\clubsuit,\end{split}
\]
with
\[
  \begin{split}
    \spadesuit&=\frac{1}{2\pi i}\int_{|\lambda|=R}\left({\rm Tr}\,\left(\sum_{\ell=0}^N
                \frac{f(\lambda)}{\lambda^{\ell+1}}\left(U_-^\ell-U_+^\ell\right)
                \right)\right)d\lambda\\
              &=\frac{1}{2\pi i}\int_{|\lambda|=R}\left(\sum_{\ell=0}^N
                \frac{f(\lambda)}{\lambda^{\ell+1}}\left({\rm Tr}\,\left(U_-^\ell
                -U_+^\ell\right)\right)\right)d\lambda\\
              &=\sum_{\ell=0}^N\left(\frac{1}{2\pi i}\int_{|\lambda|=R}
                \frac{f(\lambda)}{\lambda^{\ell+1}}d\lambda\right)
                \left({\rm Tr}\,\left(U_-^\ell-U_+^\ell\right)\right)\\
              &={\rm Tr}\,\left(\sum_{\ell=0}^Nf_\ell U_-^\ell\right)-{\rm Tr}\,
                \left(\sum_{\ell=0}^Nf_\ell U_+^\ell\right).
              \end{split}
\]

Furthermore, by \eqref{7-16},
\[
  |\clubsuit|\le\e \max_{|\lambda|=R}|f(\lambda)|
\]
which ends the proof of the step since $\e$ is arbitrary.\\

STEP 4: {\sl To conclude the proof we compute
  \begin{equation}
    \label{residue-567}
\frac{1}{2\pi i}\int_{|\lambda|=R}f(\lambda)\frac{\Phi^\prime(\lambda)}{\Phi(\lambda)}d\lambda.
    \end{equation}
}

The poles of the function $f(\lambda)\frac{\Phi^\prime(\lambda)}{\Phi(\lambda)}$ are all simple and on the unit circle, as is seen from formula \eqref{link} and the similar
formula for $\Phi^{-1}$. They correspond to the eigenvalues of $\mathcal A^\times$
and $\mathcal A$. Thus each pole, say $w$, has residue

\[
          {\rm Res}_{z=w}f(z)\frac{\Phi^\prime(z)}{\Phi(z)}=f(w).
        \]
        \end{proof}

          \section{Application to canonical discrete systems}
        \setcounter{equation}{0}
        \label{seccano}
        {\bf  Prelude:} Let ${\sigma}$ be a function analytic and contractive in the open unit disk, different from a finite Blaschke product or from a unitary constant.
        One can associate to ${\sigma}$ an infinite
        sequence of numbers $\nu_0,\nu_1\ldots\in\mathbb D$, called the Schur parameters of ${\sigma}$,  which uniquely characterizes ${\sigma}$ via the recursion
        \[
          \begin{split}
            {\sigma}_0(\lambda)&={\sigma}(\lambda),\\
            \nu_0&={\sigma}(0),\\
            {\sigma}_{\ell+1}(\lambda)&=\frac{{\sigma}_\ell(\lambda)-\nu_\ell}{\lambda(1-\overline{\nu_\ell} {\sigma}_\ell(\lambda))},\\
            \nu_{\ell+1}&=\sigma_{\ell+1}(0),\quad \ell=0,1,\ldots .
            \end{split}
          \]
          This is the celebrated Schur algorithm (see \cite{schur} for I. Schur original paper, and \cite{goh1} for a translation and a survey of the topic). The formula
          \[
            {\sigma}(\lambda)=\nu_0+\frac{\lambda(1-|\nu_0|^2)}{\overline{\nu_0}\lambda-
            \dfrac{1}{\nu_1+\dfrac{\lambda(1-|\nu_1|^2)}{\overline{\nu_1}\lambda-\cdots}}}
                \]
                (see \cite[theorem 77.1 p. 285]{wall}) allows to see that:

                \begin{prop}
In the above hypothesis. if $\nu_0,\nu_1\ldots$ are the Schur parameters of $\sigma$, then $-\nu_0,-\nu_1\ldots$ are the Schur parameters of $-\sigma$.
                  \end{prop}
                  The Carath\'eodory function
                  \begin{equation}
                    \label{vvvv}
                    \psi(\lambda)=  \frac{1-\lambda \sigma(\lambda)}{1+\lambda \sigma(\lambda)}
                  \end{equation}
                  defined (up to an inversion) in \cite[(3.7)]{MR2077215} plays an important role in the theory. Note that  replacing $\sigma$ by $-\sigma$
                  sends $\varphi$ to $\varphi^{-1}$.\\

       {\bf Canonical discrete systems:}
                   The sequence of Schur coefficients plays the role of a discrete potential in an underlying linear system, called {\sl canonical discrete system},
  which arises in particular from a discretization of the telegrapher
 equation (and  systems of differential equations arising in mathematical physics; see \cite{MR1371140,ope-iden} for the latter). When one changes the sign of the potential
 the Weyl function $N$   (whose expression inthe present setting is recalled below)              goes to $-1/N$; an illustration is given in \cite{MR1960423}. As just remarked,
 the same happens in the discrete case. In either cases the move from the Weyl function to its inverse corresponds to interchanging the role of
 two operators (self-adjoint or unitary).\\

Our starting point is an infinite sequence $-\rho_0,-\rho_1,\ldots,$ of numbers in the open unit disk, which are the
      Schur coefficients of a uniquely defined function $r({\lambda})$, analytic and contractive in the open unit disk (a Schur function).
      Since the sequence is assumed infinite, $r({\lambda})$ is not a finite Blaschke product. Throughout the section we will assume that
      \begin{equation}
        \label{wiener}
        \sum_{\ell=0}^\infty |\rho_\ell|<\infty.
\end{equation}
      The sequence $\rho_0,\rho_1,\ldots$ corresponds to $-r({\lambda})$.\smallskip

      Two objects are associated  with $r$ in a natural way: A first-order discrete system and a unitary operator colligation.  A first-order discrete system is a
      recursion of the form
      \begin{equation}
        \label{tyu89}
X_{\ell+1}({\lambda})=\begin{pmatrix}1&-\rho_\ell\\ -\overline{\rho_\ell}&1\end{pmatrix}\begin{pmatrix}{\lambda}&0\\0&1\end{pmatrix} X_\ell({\lambda}),\quad \ell=0,1,\ldots
\end{equation}
with appropriate boundary condition at $\ell=0$ or as $\ell\rightarrow\infty$; such systems appear in the theory of orthogonal polynomials on the unit circle or associated to an
Hankel operator; they also help to related the {C}arath\'eodory-{T}oeplitz and the {N}ehari extension problems.
See e.g. \cite{adamyan95,MR2000c:33001} for the latter.  See also \cite{MR2002d:47021,ag-def} for the scalar case and \cite{ag-def-vec, egl-95} for the matrix-valued case
in the rational setting.
The function $r({\lambda})$ is called reflection coefficient function. Under the hypothesis \eqref{wiener} one can associated to the system a number of other functions,
called the {\sl characteristic spectral functions}, of which we mention:
\begin{enumerate}

\item The scattering function, $s({\lambda})$.

\item The spectral function, $W({\lambda})$.

\item The Weyl coefficient function, $N({\lambda})$.

  \item The asymptotic equivalence matrix function, $V({\lambda})$.
  \end{enumerate}

      Writing
    \[
      X_0({\lambda})=\begin{pmatrix}\alpha_0({\lambda})&\beta_0({\lambda})\\
        \gamma_0({\lambda})&\delta_0({\lambda})\end{pmatrix}
    \]
    the solution to \eqref{tyu89} corresponding to the boundary condition
    \begin{equation}
      \label{cond-yui}
\lim_{\ell\rightarrow\infty}X_n({\lambda})=I_2,\quad {\lambda}\in\mathbb T.
      \end{equation}
    we can express characteristic spectral functions of the system in terms of the components of $X_0$ (compare \eqref{vvv} and \eqref{vvvv}):

    \begin{eqnarray}
      r({\lambda})&=&\frac{\beta_0}{\alpha_0}(1/{\lambda}),\\
      \label{szscattering}
      s({\lambda})&=&\frac{1}{{\lambda}}\frac{\alpha_0({\lambda}){\lambda}+\beta_0({\lambda})}{\gamma_0({\lambda}){\lambda}+\delta_0({\lambda})},\\
      N({\lambda})&=&i\frac{1-{\lambda}r({\lambda})}{1+{\lambda}r({\lambda})},\\
      \label{vvv}
      W({\lambda})&=&{\rm Im}\, N({\lambda}),\quad |{\lambda}|=1,\\
      V({\lambda})&=&\begin{pmatrix}\delta_0({\lambda})&-\frac{\beta_0({\lambda})}{{\lambda}}\\
        -{\lambda}\gamma_0({\lambda})&\alpha_0({\lambda})\end{pmatrix}.
      \end{eqnarray}

      See \cite[(2.22) p. 40]{MR2216236} for $V({\lambda})$. note that the proofs there were given for the rational case, when the coefficients $\rho_\ell$ admit a realization
      of the form
      \begin{equation}
        \rho_\ell=-ca^\ell(I_p-\Delta a^{*(\ell+1)}\Omega a^{\ell+1})^{-1}b,
      \end{equation}
      where $a,b,c$ are matrices of appropriate sizes and form a minimal triple, $\rho(a)<1$ and $\Omega$ and $\Delta$ are the unique solutions of the Stein equations
      \[
\Delta-a\Delta a^*=bb^*\quad{\rm and}\quad \Omega-a^*\Omega a=c^*c.
\]
Because tthe triple $(a,b,c)$ is minimal both $\Delta$ and $\Omega$ are strictly positive. They
satisfy $\Omega^{-1}>\Delta$. For the introduction of this class of coefficients, see \cite{MR2002d:47021}.\\

  From the linear system point of view, the reflection coefficient function seems the most important. On the other hand the scattering function, which we now define, is
  more important from the physics point of view.

  \begin{thm}
    Assuming \eqref{wiener} the system \eqref{tyu89} has a unique $\mathbb C^{2\times 2}$-valued solution with the boundary condition \eqref{cond-yui}.
    \end{thm}

  \begin{thm}
    Assuming \eqref{wiener} the system \eqref{tyu89} has a unique $\mathbb C^2$-valued solution $Y_n({\lambda})$ such that\\
    $(1)$ $\begin{pmatrix}1&-1\end{pmatrix}Y_0({\lambda})=0$, and\\
    $(2)$ $\begin{pmatrix}0&1\end{pmatrix}Y_\ell({\lambda})=1+o(\ell)$, $|{\lambda}|=1$.
    Then, it holds that
    \[
      \begin{pmatrix}1&0\end{pmatrix}Y_\ell({\lambda})={\lambda}^\ell s({\lambda})+o(\ell),\quad |{\lambda}|=1,
    \]

    where $s$ is given by \eqref{szscattering}.
  \end{thm}

We have
\begin{equation}
  \label{phi789}
\varphi({\lambda})=(1+{\lambda}\overline{r(\overline{{\lambda}})})(1-{\lambda}\overline{r(\overline{{\lambda}})})^{-1}
\end{equation}
which is related to the associated Weyl function $N({\lambda})$ via $N({\lambda})=-i\overline{\varphi(\overline{{\lambda}})}$.\\

On the other hand, (up to a Hilbert space unitary map)  $r$ can be uniquely written as
\begin{equation}
  \label{rz}
r({\lambda})=H+{\lambda}G(I_{\mathcal H}-{\lambda}T)^{-1}F,
\end{equation}
where $\mathcal H$ is a Hilbert space and the operator matrix
\begin{equation}
  \label{u+}
U_+=\begin{pmatrix} T&F\\ G&H\end{pmatrix}\,\, :\,\,\mathcal H\oplus \mathbb C\,\,\,\longrightarrow \mathcal H\oplus \mathbb C
\end{equation}
is unitary. The above mentioned unicity is guaranteed by the condition that the linear span of the ranges
\begin{equation}
  \label{tyu}
\left\{ {\rm ran}\,T^{*\ell}G^*,\, \ell=0,1,\ldots\right\}\cup\left\{ {\rm ran}\, T^{\ell}F,\, \ell=0,1,\ldots\right\}
\end{equation}
is dense in $\mathcal H$, i.e. is closely connected; see \cite{adrs}.\\

{\bf The trace formula:}
Replacing $r$ by $-r$ the operator $U_+$ is replaced by
\begin{equation}
  \label{u-}
U_-=\begin{pmatrix} T&-F\\ G&-H\end{pmatrix}\,\, :\,\,\mathcal H\oplus \mathbb C\,\,\,\longrightarrow \mathcal H\oplus \mathbb C
\end{equation}
which is also unitary.\\

\begin{prop}
The operators $U_+$ and $U_-$ defined respectively by \eqref{u+} and \eqref{u-} satisfy \eqref{ze-inter}.
\end{prop}

\begin{proof}
  We proceed in a number of steps.\\

  STEP 1: {\sl Let $f=\begin{pmatrix}u\\v\end{pmatrix}\in\mathcal H\oplus \mathbb C$ be such that $U_+f=U_-f$. Then, $v=0$.}\smallskip

  Indeed, the condition $(U_+-U_-)f=0$ amounts to
\[
  \begin{pmatrix}0&2F\\0&2H\end{pmatrix}\begin{pmatrix}u\\v\end{pmatrix}=\begin{pmatrix}0\\ 0 \end{pmatrix},
\]
and so we get $Fv=0$ and $Hv=0$. It may be that $H=0$ but $F\not=0$ for a non-constant $r({\lambda})$. From $F\not=0$ we get $F^*Fv=0$ and so $v=0$.\\

STEP 2: {\sl Let $N\in\left\{2,3,\ldots\right\}$ and assume that $f$ as above with $v=0$ belongs to $\cap_{\ell=1}^N\ker(U_+^\ell-U_-^\ell)$.
  Then $Gu=GTu=\cdots =GT^{N-2}u=0$ and
  \begin{equation}
    U_-^N\begin{pmatrix}u\\0\end{pmatrix}=\begin{pmatrix}T^{N}u\\ GT^{N-1}u\end{pmatrix}.
    \end{equation}
  }

  We proceed by induction and first check for $N=2$. We write
  \[
    \begin{split}
      0&=(U_+^2-U_-^2)\begin{pmatrix}u\\0\end{pmatrix}\\
      &=U_+\underbrace{(U_+-U_-)\begin{pmatrix}u\\0\end{pmatrix}}_{=0}+(U_+-U_-)U_-\begin{pmatrix}u\\0\end{pmatrix}\\
      &=\begin{pmatrix}0&2F\\0&2H\end{pmatrix}\begin{pmatrix}T&-F\\G&-H\end{pmatrix}\begin{pmatrix}u\\0\end{pmatrix}\\
      &=\begin{pmatrix}2FGu\\ 2HGu\end{pmatrix}
          \end{split}
  \]
  and so $Gu=0$ since $F\not=0$. Furthermore,
  \[
U_-\begin{pmatrix}u\\0\end{pmatrix}=\begin{pmatrix}Tu\\ Gu\end{pmatrix}
\]
which proves the recursion hypothesis for $N=2$. Assume now the hypothesis true at rank $N$. To prove it at rank $N+1$ we write similarly as in the above,
\[
    \begin{split}
      0&=(U_+^{N+1}-U_-^{N+1})\begin{pmatrix}u\\0\end{pmatrix}\\
      &=U_+\underbrace{(U_+^N-U_-^N)\begin{pmatrix}u\\0\end{pmatrix}}_{=0 \,\,\mbox{\rm by hypothesis}}+(U_+-U_-)U_-^N\begin{pmatrix}u\\0\end{pmatrix}\\
      &=\begin{pmatrix}0&2F\\0&2H\end{pmatrix}\begin{pmatrix}T^{N}u\\ GT^{N-1}u\end{pmatrix}\quad(\mbox{\rm hypothesis at rank $N$})\\
      &=\begin{pmatrix}2FGT^{N-1}u\\ 2HGT^{N-1}u\end{pmatrix}
          \end{split}
  \]
  and so $GT^{N-1}u=0$ since $F\not=0$. Furthermore,
  \[
    \begin{split}
      U_-^{N+1}\begin{pmatrix}u\\0\end{pmatrix}&=U_-U_-^N\begin{pmatrix}u\\0\end{pmatrix}\\
      &=U_-\begin{pmatrix}T^{N}u\\ GT^{N-1}u\end{pmatrix}\\
      &=U_-\begin{pmatrix}T^{N}u\\ 0\end{pmatrix}\\
      &=\begin{pmatrix}T^{N+1}u\\ GT^{N}u\end{pmatrix},
      \end{split}
    \]
    and the recursion hypothesis is proved for $N+1$.\\

In  a similar way one proves:\smallskip

STEP 3: {\sl Let $N\in\left\{2,3,\ldots\right\}$ and assume that $f$ as above with $v=0$ belongs to $\cap_{n=1}^N\ker(U_+^{*\ell}-U_-^{*\ell})$.
  Then $F^*u=F^*T^*u=\cdots =F^*T^{*(N-2)}u=0$ and
  \begin{equation}
    U_-^{*N}\begin{pmatrix}u\\0\end{pmatrix}=\begin{pmatrix}T^{*N}u\\- F^*T^{*(N-1)}u\end{pmatrix}.
    \end{equation}
  }

    STEP 4: {\sl In the notation of, and with the hypothesis of, step 2, we have $u=0$.}\smallskip

    From Steps 2 and 3 we have
    \[
        GT^\ell u=0\quad{\rm and}\quad        F^*T^{*\ell}u=0,\quad \ell=0,1,\ldots .
      \]
    The claim follows from the density in $\mathcal H$ of the linear span of the ranges \eqref{tyu}.\\

    STEP 5: {\sl It holds that
      \[
        \begin{split}
\left(\bigcap_{\ell=1}^\infty\ker (U_+^\ell-U_-^\ell) \right)\bigcap    \left(   \bigcap_{\ell=1}^\infty\ker (U_+^{*\ell}-U_-^{*\ell})\right)
&=\\
&\hspace{-3cm}=\bigcap_{\lambda\in\mathbb C\setminus\mathbb T}\ker\left((U_+-\lambda I_{\mathcal H})^{-1}-(U_--\lambda I_{\mathcal H})^{-1}\right).
\end{split}
\]}
This stems from
\[
(U_+-\lambda I_{\mathcal H})^{-1}-(U_--\lambda I_{\mathcal H})^{-1}=-\sum_{\ell=0}^\infty \lambda^{-\ell-1}(U_+^n-U_-^\ell), \quad |\lambda|>1,
\]
and \eqref{lala} for $|\lambda|<1$.

\end{proof}

\begin{thm}
With $\varphi({\lambda})$ given by \eqref{phi789}
\[
 {\rm det}({\lambda}I_{\mathcal H\oplus\mathbb C}-U_+)({\lambda}I_{\mathcal H\oplus\mathbb C}-U_-)^{-1}=\varphi({\lambda})
\]
and
\[
{\rm Tr}\left\{({\lambda}I_{\mathcal H\oplus\mathbb C}-U_+)^{-1}-({\lambda}I_{\mathcal H\oplus\mathbb C}-U_-)^{-1}\right\}=\frac{\varphi^\prime({\lambda})}{\varphi({\lambda})}.
\]
\end{thm}

  \begin{proof}
Using the unitarity of the block operator $U_+$ we have:
\begin{equation}
  I_{\mathcal H\oplus\mathbb C}-U_+^*U_-=U_+^*(U_+-U_-)=\begin{pmatrix}T^*&G^*\\F^*&H^*\end{pmatrix}\begin{pmatrix}0&2F\\0&2H\end{pmatrix}
  =\begin{pmatrix}0&0\\0&2\end{pmatrix},
\end{equation}
and so
\[
I_{\mathcal H\oplus\mathbb C}-U_+^*U_-=-C^*BC,
\]
with
\begin{equation}
  B=-1\quad{\rm and}\quad C=\begin{pmatrix}0&\sqrt{2}\end{pmatrix}.
\end{equation}
Note that
\[
B^{-1}+B^{-*}=-CC^*
\]
\end{proof}
\begin{thm}
  It holds that
  \begin{equation}
    \varphi({\lambda})=\left(-i\overline{N(\overline{{\lambda}})}\right)^{-1}
  \end{equation}
  where $N$ is the Weyl coefficient function of the first-order discrete system \eqref{tyu89}.
  \end{thm}
\begin{proof}
The function $\varphi$ from the model is given by formula \eqref{phi-model}, that is:
\[
  \begin{split}
\varphi({\lambda})&=B^{-1}+C(I_{\mathcal H\oplus\mathbb C}-{\lambda}U_+^*)^{-1}C^*\\
&=-1+2\begin{pmatrix}0&1\end{pmatrix}\begin{pmatrix}I_{\mathcal H}-{\lambda}T^*&-{\lambda}G^*\\ -{\lambda}F^*&1-{\lambda}H^*\end{pmatrix}^{-1}\begin{pmatrix}0\\ 1\end{pmatrix}\\
&=-1+2\begin{pmatrix}0&1\end{pmatrix}\begin{pmatrix}\star&\star\\ \star&(M({\lambda}))^{-1}\end{pmatrix}\begin{pmatrix}0\\ 1\end{pmatrix}\\
&=-1+2M({\lambda})^{-1},
\end{split}
\]
with
\begin{equation}
  M({\lambda})=1-{\lambda}H^*-{\lambda}F^*(I_{\mathcal H}-{\lambda}T^*)^{-1}{\lambda}G^*.
\end{equation}
Note that we have used the formula (see e.g. \cite[p. 18]{MR1084815})
\[
  \begin{pmatrix}A&B\\ C&D\end{pmatrix}^{-1}=\begin{pmatrix}A^{-1}+A^{-1}B(D^\square)^{-1}CA^{-1}&-A^{-1}B(D^\square)^{-1}\\
        -(D^\square)^{-1}CA^{-1}& (D^\square)^{-1}\end{pmatrix},
    \]
    with $D^\square=D-CA^{-1}B$   for the inverse of a block matrix, assuming all indicated inverses to exist.\smallskip

  It follows that, with $r({\lambda})$ given by \eqref{rz},
  \begin{equation}
\varphi({\lambda})+1=\frac{2}{1-{\lambda}\overline{r(\overline{{\lambda}})}}
\end{equation}

and so
\[
  \begin{split}
    \varphi({\lambda})&=-1+2(1-{\lambda}\overline{r(\overline{{\lambda}})})^{-1}\\
    &=(1+{\lambda}\overline{r(\overline{{\lambda}})})(1-{\lambda}\overline{r(\overline{{\lambda}})})^{-1}\\
    &=\left(\overline{-iN(\overline{{\lambda}})}\right)^{-1}.
    \end{split}
  \]
\end{proof}


\begin{thebibliography}{10}

\bibitem{adamyan95}
V.M. Adamyan and S.E. Nechayev.
\newblock Nuclear {H}ankel matrices and orthogonal trigonometric polynomials.
\newblock {\em Contemporary Mathematics}, 189:1--15, 1995.

\bibitem{MR1638044}
D.~Alpay.
\newblock {\em Algorithme de {S}chur, espaces \`a noyau reproduisant et
  th\'eorie des syst\`emes}, volume~6 of {\em Panoramas et Synth\`eses
  [Panoramas and Syntheses]}.
\newblock Soci\'et\'e Math\'ematique de France, Paris, 1998.

\bibitem{MR2002b:47144}
D.~Alpay.
\newblock {\em The {S}chur algorithm, reproducing kernel spaces and system
  theory}.
\newblock American Mathematical Society, Providence, RI, 2001.
\newblock Translated from the 1998 French original by Stephen S. Wilson,
  Panoramas et Synth\`eses.

\bibitem{CAPB_2}
D.~Alpay.
\newblock {\em An advanced complex analysis problem book. Topological Vector
  Spaces, Functional Analysis, and Hilbert spaces of analytic functions}.
\newblock Birkh\"auser/Springer Basel AG, Basel, 2015.

\bibitem{adrs}
D.~Alpay, A.~Dijksma, J.~Rovnyak, and H.~de~Snoo.
\newblock {\em {Schur} functions, operator colligations, and reproducing kernel
  {P}ontryagin spaces}, volume~96 of {\em Operator theory: {A}dvances and
  {A}pplications}.
\newblock Birkh{\" a}user Verlag, Basel, 1997.

\bibitem{ad1}
D.~Alpay and H.~Dym.
\newblock Hilbert spaces of analytic functions, inverse scattering and operator
  models, {I}.
\newblock {\em Integral Equation and Operator Theory}, 7:589--641, 1984.

\bibitem{ag}
D.~Alpay and I.~Gohberg.
\newblock Unitary rational matrix functions.
\newblock In I.~Gohberg, editor, {\em Topics in interpolation theory of
  rational matrix-valued functions}, volume~33 of {\em { Operator {T}heory:
  {A}dvances and {A}pplications}}, pages 175--222. Birkh{\" a}user Verlag,
  Basel, 1988.

\bibitem{MR2002d:47021}
D.~Alpay and I.~Gohberg.
\newblock Connections between the {C}arath\'eodory-{T}oeplitz and the {N}ehari
  extension problems: the discrete scalar case.
\newblock {\em Integral Equations Operator Theory}, 37(2):125--142, 2000.

\bibitem{MR1960423}
D.~Alpay and I.~Gohberg.
\newblock A trace formula for canonical differential expressions.
\newblock {\em J. Funct. Anal.}, 197(2):489--525, 2003.

\bibitem{MR2069002}
D.~Alpay and I.~Gohberg.
\newblock Pairs of selfadjoint operators and their invariants.
\newblock {\em Algebra i Analiz}, 16(1):70--120, 2004.

\bibitem{ag-def}
D.~Alpay and I.~Gohberg.
\newblock {Discrete analogs of canonical systems with pseudo--exponential
  potential. Definitions and formulas for the spectral matrix functions}.
\newblock In D.~Alpay and I.~Gohberg, editors, {\em {The state space method.
  New results and new applications}}, volume 161, pages 1--47. Birkh{\" a}user
  Verlag, Basel, 2006.

\bibitem{MR2216236}
D.~Alpay and I.~Gohberg, editors.
\newblock {\em Interpolation, {S}chur functions and moment problems}, volume
  165 of {\em Operator Theory: Advances and Applications}.
\newblock Birkh\"auser Verlag, Basel, 2006.
\newblock Linear Operators and Linear Systems.

\bibitem{ag-def-vec}
D.~Alpay and I.~Gohberg.
\newblock Discrete systems and their characteristic spectral functions.
\newblock {\em Mediterr. J. Math.}, 4(1):1--32, 2007.

\bibitem{MR2363355}
H.~Bart, I.~Gohberg, M.A. Kaashoek, and A.C.M. Ran.
\newblock {\em Factorization of matrix and operator functions: the state space
  method}, volume 178 of {\em Operator Theory: Advances and Applications}.
\newblock Birkh\"auser Verlag, Basel, 2008.
\newblock Linear Operators and Linear Systems.

\bibitem{MR2663312}
H.~Bart, I.~Gohberg, M.A. Kaashoek, and A.C.M. Ran.
\newblock {\em A state space approach to canonical factorization with
  applications}, volume 200 of {\em Operator Theory: Advances and
  Applications}.
\newblock Birkh\"auser Verlag, Basel; Birkh\"auser Verlag, Basel, 2010.
\newblock Linear Operators and Linear Systems.

\bibitem{dbbook}
{L. de} Branges.
\newblock {\em Espaces {H}ilbertiens de fonctions enti\`{e}res}.
\newblock Masson, {P}aris, 1972.

\bibitem{dbr1}
{L. de} Branges and J.~Rovnyak.
\newblock Canonical models in quantum scattering theory.
\newblock In C.~Wilcox, editor, {\em Perturbation theory and its applications
  in quantum mechanics}, pages 295--392. Wiley, {N}ew {Y}ork, 1966.

\bibitem{dbs}
{L. de} Branges and L.A. Shulman.
\newblock Perturbation theory of unitary operators.
\newblock {\em J. Math. Anal. Appl.}, 23:294--326, 1968.

\bibitem{MR48:904}
M.~S. Brodski{\u\i}.
\newblock {\em Triangular and {J}ordan representations of linear operators}.
\newblock American Mathematical Society, Providence, R.I., 1971.
\newblock Translated from the Russian by J. M. Danskin, Translations of
  Mathematical Monographs, Vol. 32.

\bibitem{MR2000c:33001}
A.~Bultheel, P.~Gonzalez-Vera, E.~Hendriksen, and O.~Njastad.
\newblock {\em Orthogonal rational functions}.
\newblock Cambridge University Press, Cambridge, 1999.

\bibitem{DI1}
H.~Dym and A.~Iacob.
\newblock Applications of factorization and {T}oeplitz operators to inverse
  problems.
\newblock In I.~Gohberg, editor, {\em Toeplitz centennial (Tel Aviv, 1981)},
  volume~4 of {\em Operator Theory: Adv. Appl.}, pages 233--260. Birkh\"auser,
  Basel, 1982.

\bibitem{egl-95}
R.~Ellis, I.~Gohberg, and D.~Lay.
\newblock Infinite analogues of block {T}oeplitz matrices and related
  orthogonal functions.
\newblock {\em Integral Equations Operator Theory}, 22:375--419, 1995.

\bibitem{MR525380}
P.~Faurre, M.~Clerget, and F.~Germain.
\newblock {\em Op\'erateurs rationnels positifs}, volume~8 of {\em M\'ethodes
  Math\'ematiques de l'Informatique [Mathematical Methods of Information
  Science]}.
\newblock Dunod, Paris, 1979.
\newblock Application \`a l'hyperstabilit\'e et aux processus al\'eatoires.

\bibitem{hspnw}
B.~Fritzsche and B.~Kirstein, editors.
\newblock {\em {Ausgew\"{a}hlte {A}rbeiten zu den {U}rspr\"{u}ngen der
  {S}chur--{A}nalysis}}, volume~16 of {\em {Teubner--Archiv zur Mathematik}}.
\newblock {B.G. Teubner Verlagsgesellschaft, Stuttgart--Leipzig}, 1991.

\bibitem{goh1}
I.~Gohberg, editor.
\newblock {\em I. Schur methods in operator theory and signal processing},
  volume~18 of {\em Operator theory: {A}dvances and {A}pplications}.
\newblock Birkh{\" a}user Verlag, Basel, 1986.

\bibitem{herglotz}
G.~Herglotz.
\newblock {\"{U}}ber {P}otenzenreihen mit positiven reelle {T}eil im
  {E}inheitskreis.
\newblock {\em Sitzungsber {S}achs. {A}kad. {W}iss. {L}eipzig, {M}ath},
  63:501--511, 1911.

\bibitem{MR1084815}
R.A. Horn and C.R. Johnson.
\newblock {\em Matrix analysis}.
\newblock Cambridge University Press, Cambridge, 1990.
\newblock Corrected reprint of the 1985 original.

\bibitem{kyp4}
J.~Suykens P. Van~Dooren I.~Goethals, T. Van~Gestel and B.~De Moor.
\newblock Identification of positive real models in subspace identification by
  using regularization.
\newblock 48, 2003.

\bibitem{kaash-field}
M.~Kaashoek.
\newblock State space theory of rational matrix functions and applications.
\newblock In P.~Lancaster, editor, {\em Lectures on operator theory and its
  applications}, volume~3 of {\em Fields {I}nstitute {M}onographs}, pages
  233--333. American {M}athematical {S}ociety, 1996.

\bibitem{MR1710394}
M.~G. Krein.
\newblock {\em Izbrannye trudy. {III}}.
\newblock Natsionalnaya Akademiya Nauk Ukrainy Institut Matematiki,
  Kiev, 1997.
\newblock Spektralnaya teoriya struny i voprosy ustoichivosti. [Spectral string
  theory and stability problems].

\bibitem{krein-trace2}
M.G. Krein.
\newblock On perturbation determinants and a trace formula for unitary and
  self--adjoint operators.
\newblock {\em Dokl. {A}kad. {N}auk. {SSSR}}, 144:145--152, 1962.
\newblock In {R}ussian.

\bibitem{MR86m:00014}
M.G. Krein.
\newblock {\em Topics in differential and integral equations and operator
  theory}, volume~7 of {\em Operator theory: {A}dvances and {A}pplications}.
\newblock Birkh\"auser Verlag, Basel, 1983.
\newblock Edited by I. Gohberg, Translated from the Russian by A. Iacob.

\bibitem{kyp3}
B.E. L.~Hitz and Prof. B. D.~O. Anderson.
\newblock Discrete positive-real functions and their application to system
  stability.
\newblock 116.

\bibitem{MR3861897}
C.~Liaw and S.~Treil.
\newblock General {C}lark model for finite-rank perturbations.
\newblock {\em Anal. PDE}, 12(2):449--492, 2019.

\bibitem{kyp2}
M.~Liu and J.~Xiong.
\newblock Bilinear transformation for discrete-time positive real and negative
  imaginary systems.
\newblock {\em {IEEE} {T}ransactions on automatic control}, 63, 2018.

\bibitem{Rantez}
A.~Rantzer.
\newblock On the {Kalman-Yakubovich-Popov} lemma.
\newblock {\em {Sys. Cont. Lett.}}, 28:7--10, 1996.

\bibitem{MR1371140}
A.~Sakhnovich.
\newblock Canonical systems and transfer matrix-functions.
\newblock {\em Proc. Amer. Math. Soc.}, 125(5):1451--1455, 1997.

\bibitem{ope-iden}
L.A. Sakhnovich.
\newblock {\em Spectral theory of canonical differential systems. {M}ethod of
  operator identities}, volume 107 of {\em Operator Theory: Advances and
  Applications}.
\newblock Birkh\"auser Verlag, Basel, 1999.
\newblock Translated from the Russian manuscript by E. Melnichenko.

\bibitem{schur}
I.~Schur.
\newblock {\"U}ber die {P}otenzreihen, die im {Innern des Einheitkreises
  beschr{\" a}nkten sind, {I}}.
\newblock {\em Journal f\"ur die Reine und Angewandte Mathematik},
  147:205--232, 1917.
\newblock {\rm English translation in: I. Schur methods in operator theory and
  signal processing. (Operator theory: {A}dvances and {A}pplications OT 18
  (1986), Birkh{\" a}user Verlag), Basel}.

\bibitem{MR2077215}
B.~Simon.
\newblock Analogs of the {$m$}-function in the theory of orthogonal polynomials
  on the unit circle.
\newblock {\em J. Comput. Appl. Math.}, 171(1-2):411--424, 2004.

\bibitem{Stone}
M.~Stone.
\newblock {\em Linear transformations in {H}ilbert space}.
\newblock American {M}athematical {S}ociety, {P}rovidence, {R}hode {I}sland,
  1932.

\bibitem{szego}
G.~Szeg{\"{o}}.
\newblock {\em Orthogonal polynomials}, volume~23.
\newblock {Amer.} Math. Soc., {Rhodes} Island, 1975.

\bibitem{zbMATH03263403}
A.~G. Vitushkin.
\newblock The analytic capacity of sets in problems in approximation theory.
\newblock {\em Russ. Math. Surv.}, 22(6):139--200, 1967.

\bibitem{wall}
H.S. Wall.
\newblock {\em Analytic theory of continued fractions}.
\newblock Van {N}ostrand, 1948.

\bibitem{kyp1}
C.~Xiao and D.~Hill.
\newblock Generalizations and new proof of the discrete-time positive real
  lemma and bounded real lemma.
\newblock {\em {IEEE} Transactions on circuits and systems-{I}: Fundamental
  theory and applications}, 46, 1999.

\end{thebibliography}
\end{document}